\documentclass[12pt]{article}

\usepackage{amsthm,amsmath,amssymb,enumerate,mathtools}
\usepackage[pdftex,pagebackref=false, hidelinks]{hyperref}
\newtheorem{theorem}{Theorem}[section]
\newtheorem{claim}{Claim}[theorem]
\newtheorem{lemma}[theorem]{Lemma}
\newtheorem{problem}[theorem]{Problem}
\newtheorem{proposition}[theorem]{Proposition}
\newtheorem{corollary}[theorem]{Corollary}
\newtheorem{conjecture}[theorem]{Conjecture}
\theoremstyle{definition}
\newtheorem{definition}[theorem]{Definition}
\newtheorem{remark}[theorem]{Remark}
\newtheorem{example}[theorem]{Example}

\newcommand{\bF}{\mathbb F}
\newcommand{\bR}{\mathbb R}
\newcommand{\bC}{\mathbb C}

\newcommand{\cA}{\mathcal{A}}

\newcommand{\cF}{\mathcal{F}}

\newcommand{\cH}{\mathcal{H}}
\newcommand{\cM}{\mathcal{M}}
\newcommand{\cN}{\mathcal{N}}

\newcommand{\cX}{\mathcal{X}}

\DeclareMathOperator{\si}{si}

\DeclareMathOperator{\cl}{cl}

\newcommand{\elem}{\varepsilon}
\newcommand{\del}{\!\setminus\!}

\usepackage{tikz, comment, authblk}
\usepackage{cleveref}

\title{Excluding a line from positroids}
\author[1]{Jonathan Boretsky\thanks{The first author is partially supported by NSERC grant PDF 587781. Le premier auteur est partiellement soutenu par la bourse CRSNG PDF 587781.}}
\author[2]{Zach Walsh\thanks{The second author was supported in part by NSF grant DMS-2452015.}}
\affil[1]{\footnotesize McGill University, Department of Mathematics and Statistics, Montreal, QC, Canada}
\affil[2]{\footnotesize Auburn University, Department of Mathematics and Statistics, Auburn, AL, USA}
\date{\today}

\begin{document}

\maketitle

\begin{abstract}
For all positive integers $\ell$ and $r$, we determine the maximum number of elements of a simple rank-$r$ positroid without the rank-$2$ uniform matroid $U_{2,\ell+2}$ as a minor, and characterize the matroids with the maximum number of elements.
We prove this as a consequence of a more general result, which also determines the maximum number of elements of a simple rank-$r$ bicircular matroid, lattice path matroid, multi-path matroid, or colaminar matroid with no $U_{2,\ell+2}$-minor.
This result continues a long line of research into upper bounds on the number of elements of matroids from various classes that forbid $U_{2,\ell+2}$ as a minor. This is the first paper to study positroids in this context, and it suggests methods to study similar problems for other classes of matroids, such as gammoids or base-orderable matroids.
\end{abstract}

\section{Introduction} \label{sec: introduction}

Extremal matroid theory was formally introduced in 1993 with the following problem posed by Joseph Kung \cite{Kung-1993}.

\begin{problem} \label{prob: Kung problem}
Let $\cM$ be a class of matroids. 
Determine the maximum number of elements of a simple rank-$r$ matroid in $\cM$, and characterize the matroids for which equality holds.
\end{problem}

This maximum does not exist for all classes of matroids. 
For instance, there is no finite answer to Problem~\ref{prob: Kung problem} for the class of all matroids, or for any class containing all uniform matroids.
However, Kung \cite[Theorem 4.3]{Kung-1993} proved that every simple rank-$r$ matroid with no $U_{2, \ell+2}$-minor has at most $\frac{\ell^r - 1}{\ell - 1}$ elements.
If $\ell$ is a prime power, then this bound is tight and is achieved by the rank-$r$ projective geometry over the finite field of order $\ell$. It is thus natural to study Problem~\ref{prob: Kung problem} for classes of matroids without $U_{2,\ell}$-minors. 

In this paper, we study the class of \textit{positroids} without $U_{2,\ell}$-minors. Positroids, which are matroids representable by a matrix over $\bR$ with all maximal subdeterminants nonnegative, occupy a central position at the intersection of combinatorics, algebraic geometry, and physics. They serve as combinatorial models for the totally nonnegative Grassmannian, which in turn underlies the construction of the amplituhedron in quantum field theory. The study of amplituhedra and related objects remains an active and influential area in theoretical physics \cite{AHT}. Through these connections, positroids emerge as unifying objects that integrate discrete combinatorial frameworks with geometric and physical theory. Combinatorially, they are remarkably rich: Postnikov~\cite{postnikov} showed that positroids are in bijection with decorated permutations, Grassmann necklaces, plabic graphs, and \reflectbox{L}-diagrams, providing multiple equivalent viewpoints.
Positroids also form a minor-closed class of matroids. It is thus natural to ask extremal questions about positroids and their minors. There has been work enumerating both positroids \cite{Williams05} and connected positroids \cite{ARW16}, but as far as the authors are aware, no other extremal results about positroids are known. Our main result tightly resolves Problem~\ref{prob: Kung problem} for the class of $U_{2,\ell+2}$-minor-free positroids.

\begin{theorem} \label{thm: main}
For all integers $r, \ell \ge 1$, if $M$ is a simple rank-$r$ positroid with no $U_{2,\ell+2}$-minor, then $|M| \le \ell(r - 1) + 1$.
Moreover, if $r \ge 2$ then equality holds if and only if $M$ can be obtained by taking parallel connections of copies of $U_{2, \ell+1}$.
\end{theorem}

We will obtain Theorem \ref{thm: main} as a corollary of a stronger result (Theorem \ref{thm: main result with proof}) that also gives a sharp bound for other well-studied minor-closed $M(K_4)$-free classes of $\bR$-representable matroids; we will define these classes in Section 6.

\begin{theorem} \label{thm: main other classes}
Let $r$ and $\ell$ be positive integers and let $\cM$ be one of the following classes of matroids: lattice path, multi-path, colaminar, bicircular, or path-circular.
The the maximum number of elements of a simple rank-$r$ matroid in $\cM$ with no $U_{2,\ell+2}$-minor is $\ell(r - 1) + 1$.
\end{theorem}

Theorem \ref{thm: main} lies at the intersection of three areas of research: extremal matroid theory, structural properties of positroids, and structural properties of matroids with no $M(K_4)$-minor.
We will explain how Theorem \ref{thm: main} continues lines of research in these three areas, beginning with extremal matroid theory.
The following beautiful result of Geelen and Nelson \cite[Theorem 1.2]{Geelen-Nelson-2010} generalizes the initial work of Kung and tightly solves Problem \ref{prob: Kung problem} for the class of matroids with no $U_{2,\ell+2}$-minor, for all values of $\ell$ and all sufficiently large values of $r$.

\begin{theorem}[Geelen, Nelson \cite{Geelen-Nelson-2010}] \label{thm: exlude a line}
Let $\ell\ge 2$ be an integer and let $q$ be the largest prime power at most $\ell$.
For $r$ sufficiently large, every simple rank-$r$ matroid with no $U_{2,\ell+2}$-minor has at most $\frac{q^r - 1}{q - 1}$ elements, with equality only for the rank-$r$ projective geometry over the order-$q$ finite field.
\end{theorem}

Projective geometries are not representable over fields of characteristic zero (see \cite[page 205, Exercise 8]{Oxley-2011}), inviting the following question: for a field $\bF$ with characteristic zero, what is the maximum number of elements of a simple rank-$r$ $\bF$-representable matroid with no $U_{2, \ell+2}$-minor?
This was answered by Geelen, Nelson, and Walsh \cite[Theorem 1.5]{Geelen-Nelson-Walsh-2024} in the case that $\bF = \bC$.

\begin{theorem}[Geelen, Nelson, Walsh \cite{Geelen-Nelson-Walsh-2024}] \label{thm: complex exclude a line}
Let $\ell\ge 2$ be an integer and let $r$ be a sufficiently large integer.
Then every simple rank-$r$ $\bC$-representable matroid with no $U_{2,\ell+2}$-minor has at most $(\ell-1)\binom{r}{2} + r$ elements, with equality only for the rank-$r$ Dowling geometry over the cyclic group of order $\ell-1$.
\end{theorem}

When $\bF \in \{\bR, \mathbb Q\}$, there is a lower bound of $2\binom{r}{2} + r$ and an upper bound of $2 \binom{r}{2} + f(\ell)\cdot r$ with $f$ a function of $\ell$, due to Geelen, Nelson, and Walsh \cite[Theorem 1.6]{Geelen-Nelson-Walsh-2024}, but no exact solution to Problem~\ref{prob: Kung problem} is known for all $\ell$. 
Theorem \ref{thm: main} can be seen as a further specialization of Theorems \ref{thm: exlude a line} and \ref{thm: complex exclude a line} to the subclass of $\bR$-representable matroids consisting of positroids.

Theorem \ref{thm: main} also continues a recent line of research that studies positroids from a structural perspective.
For example, Bonin \cite{Bonin-2024} studied matroid amalgams of positroids, Quail \cite{Quail-2024} studied positroids representable over the finite fields of order $3$ and of order $4$, Quail and Rombach \cite{Quail-Rombach-2025} studied graphic positroids, and Park \cite{Park-2023} studied excluded minors for positroids.
Our proof of Theorem \ref{thm: main} relies on two new structural properties (Propositions \ref{prop: rank 3 excluded minor} and \ref{prop: rank 4 excluded minor}) for low-rank positroids.

Finally, \Cref{thm: main}, along with Theorem \ref{thm: main other classes}, contributes to a line of research studying matroids with no $M(K_4)$-minor, where $M(K_4)$ is the graphic matroid of the complete graph $K_4$. 
For example, Sims \cite{Sims-1977} proved that the class of $M(K_4)$-minor-free matroids is complete, Brylawski \cite{Brylawski-1971} characterized the binary matroids with no $M(K_4)$-minor, and Oxley \cite{Oxley-1987} characterized the $3$-connected ternary matroids with no $M(K_4)$-minor.
More recently, work of Pendavingh and van der Pol \cite{Pendavingh-vanderPol-2015} and van der Pol \cite{vanderPol-2023} investigated the importance of $M(K_4)$ in matroid enumeration.
Most notably for this paper, Kung studied Problem \ref{prob: Kung problem} for the class of matroids with no $M(K_4)$-minor and no $U_{2, \ell+2}$-minor, proving that every such simple rank-$r$ matroid has at most $(6\ell^{\ell - 1} + 8\ell)\cdot r$ elements \cite[Theorem 1.1]{Kung-1988}. The graphic matroid $M(K_4)$ is not a positroid, so this also bounds from above the maximum number of elements of a simple rank-$r$ positroid with no $U_{2, \ell+2}$-minor. 
To the best of our knowledge, this was in fact the previous best upper bound for this quantity.
There are several other notable $M(K_4)$-minor-free classes of matroids that contain the class of positroids, such as gammoids, strongly base-orderable matroids, and $k$-base-orderable matroids.
We will discuss these further in Section \ref{sec: future work}.

The rest of this paper is structured as follows.
In Section \ref{sec: preliminaries}, we will give more background on matroids and positroids.
In Section \ref{sec: excluded minors}, we will use a list of $9$ known excluded minors for the class of positroids to prove two structural properties for low-rank positroids.
In Section \ref{sec: the extremal examples}, we will describe the matroids for which equality holds in Theorem \ref{thm: main} and prove that they are in fact positroids and are $U_{2, \ell+2}$-minor-free.
In Section \ref{sec: the main proof} we will prove Theorem \ref{thm: main} as a consequence of Theorem \ref{thm: main result with proof}. Then in Section \ref{sec:minor-closed-classes}, we will comment on specializations of Theorem \ref{thm: main result with proof}  to various minor-closed classes of matroids that are contained in the positroids or in the transversal matroids, and prove Theorem \ref{thm: main other classes}. We conclude in Section \ref{sec: future work} by discussing several directions for future work.

\section{Preliminaries} \label{sec: preliminaries}

\subsection{Matroids}

In this section, we recall relevant concepts from matroid theory. We assume basic knowledge of matroids and refer the reader to \cite{Oxley-2011} for more details. We focus on the perspective of \emph{flats}, which will be particularly useful for our work.

\begin{definition}[Matroids in terms of flats]
Let $E$ be a finite set. A \emph{matroid} is a pair $M = (E, \cF)$ where $\mathcal{F} \subseteq 2^E$ satisfies the following axioms:
\begin{enumerate}
    \item[(F1)] $E \in \mathcal{F}$.
    \item[(F2)] For all $F_1, F_2 \in \mathcal{F}$, we have $F_1 \cap F_2 \in \mathcal{F}$.
    \item[(F3)] For each $F \in \mathcal{F}$ and each $e \in E - F$, there exists a unique $F' \in \mathcal{F}$ that is minimal with respect to inclusion among elements of $\mathcal{F}$ containing $F \cup \{e\}$.
\end{enumerate}
The set $\cF$ is the set of \emph{flats} of $M$. 
\end{definition}

\begin{remark}
Flats are highly structured: ordered by inclusion, they form a lattice called the \emph{matroid lattice}. Given a collection of flats $\mathcal{F} \subseteq 2^E$ satisfying the axioms above, one can recover other familiar data of the matroid $M = (E, \cF)$:
\begin{itemize}
    \item The \emph{rank} of the matroid, denoted $r(M)$, is the length of a longest chain of flats in the lattice.
    \item The \emph{bases} are subsets $B \subseteq E$ of size $r(M)$ which are not contained in any flat other than $E$ itself, and the \emph{independent sets} are the subsets of bases.
     \item The \emph{closure} of a set $X$, denoted $\cl(X)$, is the join of all flats contained in $X$. Explicitly, this is the minimal flat containing $X$.
\end{itemize}
\end{remark}

The flats of a matroid $M = (E, \cF)$ of rank $1$, $2$, and $3$ are the \emph{points}, \emph{lines}, and \emph{planes} of $M$, respectively.
We call the lines that contain at least three points \textit{long lines}; these will play a distinguished role in what follows. 
We write $|M|$ for $|E|$ and $\elem(M)$ for the number of points of $M$.

\begin{example}
    An important matroid in this paper is the rank-$2$ uniform matroid on $\ell$ elements, denoted $U_{2,\ell}$. This is the matroid where the flats are the empty set, singletons, and the entire ground set. 
    The corresponding lattice is illustrated in \Cref{fig:uniform matroid lattice}. The bases of this matroid are thus all the 2-element subsets of the ground set. 
\end{example}

\begin{figure}[h]
\centering
\begin{tikzpicture}[every node/.style={draw, minimum size=10mm, inner sep=0.1mm}, scale=0.85]

\node (empty) at (0,0) {$\emptyset$};
\node (1) at (-3.5,2) {$\{1\}$};
\node (2) at (-1.5,2) {$\{2\}$};
\node (n-1) at (1.5,2) {$\{\ell-1\}$};
\node (n) at (3.5,2) {$\{\ell\}$};
\node[draw=none] (dots) at (0,2) {$\cdots$};
\node (all) at (0,4) {$\{1,\dots,\ell\}$};

\draw (empty) -- (1);
\draw (empty) -- (2);
\draw (empty) -- (n-1);
\draw (empty) -- (n);

\draw (1) -- (all);
\draw (2) -- (all);
\draw (n-1) -- (all);
\draw (n) -- (all);

\end{tikzpicture}
\caption{Lattice of flats of the uniform matroid $U_{2,\ell}$.}
\label{fig:uniform matroid lattice}
\end{figure}

Perhaps the canonical examples of matroids come from matrices.
If $A$ is a matrix over a field $\bF$ with column set $E$, then the set $\cF$ of maximal subsets of columns spanning each subspace of the column span of $A$ is the set of flats of a matroid on $E$.
Such a matroid is \emph{representable} over $\bF$.
In this case the bases of the matroid are the subsets of columns that are bases of the column span, and the rank of the matroid is the rank of the matrix.

Matroids can also be defined using affine independence over a field.

\begin{definition}[Simple affine matroids]
If $E \subseteq \bF^{r}$, then the set of affinely independent subsets of $E$ is the collection of independent sets of a simple matroid $M$ on $E$. 
Such a matroid is \emph{affine} over $\bF$. 
The lines and planes of $M$ can be understood geometrically as follows:
\begin{itemize}
    \item If $r(M) \ge 2$, the lines of $M$ are the maximal subsets of $E$ which lie on a line of $\mathbb{F}^r$.

\item If $r(M) \ge 3$, the planes of $M$ are the maximal subsets of $E$ which lie on a plane of $\mathbb{F}^r$.
    
\end{itemize}
\end{definition}

In practice, when $\mathbb{F}=\mathbb{R}$, we represent $M$ by drawing $E$ in $\mathbb{R}^r$, together with the lines and planes of $\mathbb{R}^r$ containing the lines and planes of $M$. We often do not draw lines through exactly two points, nor planes through the union of a (not necessarily long) line and a point, because the existence of such lines and planes is implied by construction.

\begin{example}
Consider the $9$-element subset of $\bR^3$ shown in \Cref{fig: rank-4 excluded minor}, and let $M$ be the corresponding affine matroid.
We have explicitly drawn the four long lines of $M$.
There are also many $2$-point lines of $M$ not explicitly drawn that connect pairs of points not on a long line.  
These are all the lines of $M$. 
We have explicitly drawn the three planes of $M$ that are not the union of a line and a point.
There are also many planes of $M$ not explicitly drawn, which each consists of the union of a line and a point. 
\end{example}

We will need two matroid operations, \emph{restriction} and \emph{contraction}, in this paper, each of which we can define in terms of the lattice of flats of the matroid.
Let $M$ be a matroid on a ground set $E$ with flat set $\mathcal{F}$, and let $A \subseteq E$. The \emph{restriction} of $M$ to $A$, denoted $M|A$, is the matroid on $A$ whose flats are exactly the intersections $F \cap A$ for $F \in \mathcal{F}$. That is,
\[
\mathcal{F}(M|A) = \{ F \cap A \mid F \in \mathcal{F} \}.
\]
If $A=F'$ happens to be a flat, then we have that \[\mathcal{F}(M|F')=\{F \mid F\in \mathcal{F}, \,F\subseteq F'\}.\]
The \emph{rank} of $A$ in $M$ is the rank of the matroid $M|A$.
Intuitively, restriction “removes” elements not in $A$ while preserving the lattice structure among elements of $A$.
We will sometimes also write $M \del (E - A)$ for $M|A$; this is the \emph{deletion} of $E - A$ from $M$. Dually, the \emph{contraction} of $M$ by $F'\in\mathcal{F}$, denoted $M/F'$, is the matroid on $E - F'$ whose flats are the flats of $M$ containing $F'$, with $F'$ removed:
\[
\mathcal{F}(M/F') = \{ F - F' \mid F \in \mathcal{F},\, F' \subseteq F \}.
\]
Contraction can be thought of as “collapsing” $F'$ to the empty set, retaining the lattice structure among elements outside of $F'$.

Both operations interact nicely with affine representations: restriction to $A$ corresponds to deleting points not in $A$, while contraction of a flat $F'$ corresponds to projecting the matroid along the subspace spanned by $F'$, preserving incidences among the remaining points, lines, and planes. These perspectives will allow us to manipulate matroids in a way that is compatible with both their flats and their geometric structures.

This paper studies classes of matroids that are closed under deletion and contraction.

\begin{definition}[Minors]\label{Def:matroid minor}
Let $M$ be a matroid. A \emph{minor} of $M$ is any matroid obtained from $M$ by a sequence of deletions and contractions.
A class $\mathcal{C}$ of matroids is \emph{minor-closed} if, whenever $M \in \mathcal{C}$ and $N$ is isomorphic to a minor of $M$, we have $N \in \mathcal{C}$. Equivalently, $\mathcal{C}$ is closed under restriction, contraction, and isomorphism.
\end{definition}

A \emph{separator} of a matroid $M = (E, \cF)$ is a set $X \subseteq E$  so that $r(X) + r(E - X) = r(M)$.
A matroid is \emph{connected} if it has no separators other than $\emptyset$ and $E$.
The \emph{components} of $M$ are its maximal connected restrictions.

Consider the matroid on $E=[n]$ with flats $\emptyset$, \{1\}, $[n] - \{1\}$, and $[n]$.
The bases of this matroid are $\{\{1,i\}\mid 2\leq i\leq n\}$. 
This is a matroid on $n$ elements with no $U_{2,3}$-minor. 
In order to make Problem \ref{prob: Kung problem} interesting, we must avoid pathological examples like this one. 
We do so by restricting our attention to simple matroids.

\begin{definition}[Simple matroids]\label{def:SimpleMatroid}
A matroid $M$ is \emph{simple} if every element is a flat.
Equivalently, $M$ is simple if each subset with at most two elements is independent.
\end{definition}

A dependent element of a matroid is a \emph{loop}, and a dependent pair of non-loops is a \emph{parallel pair}.

We will need the fact that every matroid has a corresponding simple matroid.

\begin{definition}[Simplification]\label{def:Simplification}
Let $M$ be a matroid on ground set $E$.  
The \emph{simplification} of $M$, denoted $\operatorname{si}(M)$, is the simple
matroid obtained by deleting all loops of $M$ and identifying each parallel
class of $M$ to a single representative element.  
Equivalently, $\operatorname{si}(M)$ is the unique simple
matroid whose lattice of flats is obtained from the lattice of flats of $M$ by deleting all elements of the rank-$0$ flat of $M$ and applying the natural projection that collapses each rank-$1$ flat of $M$ to a single element.
\end{definition}

Note that the elements of $\si(M)$ are in bijection with the points of $M$, so $|\si(M)| = \elem(M)$.

\subsection{Positroids}

Positroids are a special class of matroids that arise naturally in the study of the totally nonnegative Grassmannian and have remarkable combinatorial structure. They are $\bR$-representable matroids whose representations satisfy additional positivity constraints.

\begin{definition}[Positroids]
        A matroid is a \emph{positroid} if it can be represented by a full-rank real matrix such that all maximal subdeterminants of the matrix are nonnegative.
\end{definition}

\begin{example}
    The matroid on the ground set $\{1,2,3,4\}$ with bases $\{1,2\}$, $\{1,3\}$, $\{1,4\}$, $\{2,3\}$, and $\{2,4\}$ is a positroid, represented by the matrix $\begin{pmatrix}
        1 & 1 & 1 &1 \\
        0 & 1 & 2 &2
    \end{pmatrix}.$ 
    
    For a non-example, consider the matroid in the middle left of \Cref{fig: rank-3 excluded minors}. This is the graphic matroid of the complete graph $K_4$, denoted $M(K_4)$. Indeed, one can verify that any $\mathbb{R}$-representation of $M(K_4)$ has at least one negative maximal subdeterminant.
\end{example}


In line with our earlier introduction to matroid flats, we give a description of positroids in terms of their flats, due to Bonin \cite[Theorem 1.2]{Bonin-2024}.

\begin{theorem}[Bonin \cite{Bonin-2024}] \label{thm:Bonin's condition}
    A connected matroid $M$ on a ground set $E$ is a positroid if and only if there is a total ordering $<$ of $E$ such that for every flat $F$ for which both $M|F$ and $M/F$ are connected, either $F$ or $E-F$ is an interval in $<$ order. 
\end{theorem}

\begin{example}\label{ex:uniform}
    The only non-empty flats of $U_{2,\ell}$ are singletons and the entire ground set, which satisfy the conditions of \Cref{thm:Bonin's condition} for any ordering. Thus, $U_{2,\ell}$ is a positroid.
\end{example}

\section{Excluded minors} \label{sec: excluded minors}

Minor-closed classes of matroids are important in matroid theory, as many natural classes of matroids are minor-closed. For instance, graphic matroids, cographic matroids, regular matroids, and matroids representable over a field $\mathbb{F}$ are all minor-closed classes.

\begin{figure}[htb] 
\[
 \begin{tikzpicture}[scale=0.3]
 
  \filldraw (-4,0) circle (10pt);
  \filldraw (0,0) circle (10pt);
  \filldraw (4,0) circle (10pt);
  
  \filldraw (-2,4) circle (10pt);
  \filldraw (0,4) circle (10pt);
  \filldraw (4,4) circle (10pt);

  \filldraw (0,8) circle (10pt);

  \draw[thick] (-4,0) -- (4,0);
  \draw[thick] (-4,0) -- (0,8);
  \draw[thick] (0,0) -- (0,8);
 \end{tikzpicture}
 \endpgfgraphicnamed
 \]
 \caption{A rank-$3$ excluded minor for the class of positroids.} 
 \label{fig: a rank-3 excluded minor}
\end{figure}

In this section we will use some known results for \emph{excluded minors} for the class of positroids, which are minor-minimal matroids that are not positroids.
Since positroids are precisely the base-sortable matroids \cite{positroid-is-base-orderable}, the rank-$3$ excluded minors for the class of positroids have been completely classified by Blum \cite[Corollary 4.12]{Blum-2001}.
While there are infinitely many rank-$3$ excluded minors, we use the subset shown in Figures \ref{fig: a rank-3 excluded minor} and \ref{fig: rank-3 excluded minors} to prove that rank-$3$ matroids satisfying a certain structural property are not positroids.

\begin{proposition}\label{prop: rank 3 excluded minor}
    If $M$ is a simple rank-$3$ matroid with a line $L$ and distinct points $e_1, e_2, e_3$ on $L$ each on at least two long lines, then $M$ is not a positroid.
    In particular, $M$ has a minor from Figure \ref{fig: rank-3 excluded minors}.
\end{proposition}
\begin{proof}
Let $M_1, M_2, \dots, M_9$ be the matroids of Figure \ref{fig: rank-3 excluded minors}, reading from left to right and top to bottom.
Let $M$ be deletion-minimal with a long line $L$ and distinct points $e_1,e_2,e_3$ on $L$ so that for $i = 1,2,3$ the point $e_i$ is on a long line $L_i$ other than $L$.
Then $L = \{e_1, e_2, e_3\}$ and $|L_i| = 3$ for $i = 1,2,3$, so $|M| \le 9$.

First suppose that $L_1, L_2, L_3$ are pairwise disjoint.
If $M$ has only four long lines, then $M \cong M_1$.
If $M$ has a fifth long line $L'$ so that $L' \cap L = \{e_i\}$ for some $i \in \{1,2,3\}$, by replacing $L_i$ with $L'$ and deleting $L_i-\{e_i\}$, we see that $M$ is not deletion-minimal, a contradiction.
So if $M$ has exactly five long lines then $M \cong M_2$.
Suppose that, for some $i\in\{1,2,3\}$, there exists $e \in L_i - L$ lying on at least three long lines, that is, on $L_i$ and at least two other long lines $L_4$ and $L_5$.
Since $L_4$ is distinct from $L_j$ for $j\in \{1,2,3\}$, it must intersect each such line in at most 1 point and consequently in exactly one point. 
Let $f \in L_i - \{e_i, e\}$. 
Then in $M \del f$, each point on $L_4$ is on a long line other than $L_4$: $e$ is on $L_5$ and the other two points are on $L_j$ for some $j\in\{1,2,3\} - \{i\}$.
Therefore $M$ is not deletion-minimal, a contradiction.
Thus, the long lines of $M$ other than $L_1, L_2, L_3$ are pairwise disjoint, so $M$ has at most six long lines, and if $M$ has six long lines then $M \cong M_3$.

Next suppose that each pair from $\{L_1, L_2, L_3\}$ shares a common point.
If these three lines do not share a common point, then $M \cong M_4$.
So they share a common point $e$.
If $M$ has only four long lines then $M \cong M_5$.
If $M$ has a fifth long line $L'$ with $L' \cap L = \{e_i\}$ for some $i \in \{1,2,3\}$, then by replacing $L_i$ with $L'$ and deleting $L_i-\{e_i,e\}$, we contradict the minimality of $M$.
It follows that $M$ has at most five long lines, and if $M$ has a fifth long line then $M \cong M_6$.

Finally, suppose that $L_i$ and $L_j$ intersect and $L_k$ does not intersect $L_i$, where $i,j,k \in \{1,2,3\}$ are distinct; we may assume that $(i,j,k) = (1,2,3)$.
Suppose $L_3$ intersects $L_2$.
If $M$ has only four long lines then $M \cong M_7$.
If $M$ has a fifth long line $L'$, then $L'$ consists of $e_2$, the unique point of $L_1 - (L \cup L_2)$, and the unique point of $L_3 - (L \cup L_2)$, so $M$ has only five long lines and we see that $M \cong M_6$.
So we may assume that $L_3$ is disjoint from $L_1 \cup L_2$.
If $M$ has only four long lines then $M \cong M_8$, so $M$ has a fifth line $L'$.
If $L' \cap L = \{e_i\}$ for $i \in \{1,2,3\}$, then by replacing $L_i$ with $L'$ we contradict the minimality of $M$.
So $L' \cap L = \emptyset$, and it follows that $M$ has exactly five lines, and $M \cong M_9$.
\end{proof}

\begin{figure}[htb] 
\[
\renewcommand{\arraystretch}{1.2}
\setlength{\tabcolsep}{12pt}
\begin{tabular}{c c c}

 \begin{tikzpicture}[scale=0.3]
 
  \filldraw (-4,0) circle (10pt);
  \filldraw (0,0) circle (10pt);
  \filldraw (4,0) circle (10pt);
  
  \filldraw (-4,4) circle (10pt);
  \filldraw (0,4) circle (10pt);
  \filldraw (4,4) circle (10pt);

  \filldraw (-4,8) circle (10pt);
  \filldraw (0,8) circle (10pt);
  \filldraw (4,8) circle (10pt);

  \draw[thick] (-4,0) -- (4,0);
  \draw[thick] (-4,0) -- (-4,8);
  \draw[thick] (0,0) -- (0,8);
  \draw[thick] (4,0) -- (4,8);
  
 \end{tikzpicture}
 
 & 
 
  \begin{tikzpicture}[scale=0.3]
 
  \filldraw (-4,0) circle (10pt);
  \filldraw (0,0) circle (10pt);
  \filldraw (4,0) circle (10pt);
  
  \filldraw (-4,4) circle (10pt);
  \filldraw (0,4) circle (10pt);
  \filldraw (4,4) circle (10pt);

  \filldraw (-4,8) circle (10pt);
  \filldraw (0,8) circle (10pt);
  \filldraw (4,8) circle (10pt);

  \draw[thick] (-4,0) -- (4,0);
  \draw[thick] (-4,0) -- (-4,8);
  \draw[thick] (0,0) -- (0,8);
  \draw[thick] (4,0) -- (4,8);
  \draw[thick] (-4,4) -- (4,4);
  
 \end{tikzpicture}
 
 & 
 
  \begin{tikzpicture}[scale=0.3]
 
  \filldraw (-4,0) circle (10pt);
  \filldraw (0,0) circle (10pt);
  \filldraw (4,0) circle (10pt);
  
  \filldraw (-4,4) circle (10pt);
  \filldraw (0,4) circle (10pt);
  \filldraw (4,4) circle (10pt);

  \filldraw (-4,8) circle (10pt);
  \filldraw (0,8) circle (10pt);
  \filldraw (4,8) circle (10pt);

  \draw[thick] (-4,0) -- (4,0);
  \draw[thick] (-4,0) -- (-4,8);
  \draw[thick] (0,0) -- (0,8);
  \draw[thick] (4,0) -- (4,8);
  \draw[thick] (-4,4) -- (4,4);
  \draw[thick] (-4,8) -- (4,8);
  
 \end{tikzpicture}

  \\
$M_1$ & $M_2$ & $M_3$
 \\
 \\

 \begin{tikzpicture}[scale=0.3]
 
  \filldraw (-4,0) circle (10pt);
  \filldraw (0,0) circle (10pt);
  \filldraw (4,0) circle (10pt);
  
  \filldraw (-2,4) circle (10pt);
  \filldraw (0,3-1/3) circle (10pt);

  \filldraw (0,8) circle (10pt);

  \draw[thick] (-4,0) -- (4,0);
  \draw[thick] (-4,0) -- (0,8);
  \draw[thick] (0,0) -- (0,8);
  \draw[thick] (4,0) -- (-2,4);
 \end{tikzpicture}

 &

  \begin{tikzpicture}[scale=0.3]
 
  \filldraw (-4,0) circle (10pt);
  \filldraw (0,0) circle (10pt);
  \filldraw (4,0) circle (10pt);
  
  \filldraw (-2,4) circle (10pt);
  \filldraw (0,4) circle (10pt);
  \filldraw (2,4) circle (10pt);

  \filldraw (0,8) circle (10pt);

  \draw[thick] (-4,0) -- (4,0);
  \draw[thick] (-4,0) -- (0,8);
  \draw[thick] (0,0) -- (0,8);
  \draw[thick] (4,0) -- (0,8);
 \end{tikzpicture}
 
&

 \begin{tikzpicture}[scale=0.3]
 
  \filldraw (-4,0) circle (10pt);
  \filldraw (0,0) circle (10pt);
  \filldraw (4,0) circle (10pt);
  
  \filldraw (-2,4) circle (10pt);
  \filldraw (0,4) circle (10pt);
  \filldraw (2,4) circle (10pt);

  \filldraw (0,8) circle (10pt);

  \draw[thick] (-4,0) -- (4,0);
  \draw[thick] (-4,0) -- (0,8);
  \draw[thick] (0,0) -- (0,8);
  \draw[thick] (4,0) -- (0,8);
  \draw[thick] (-2,4) -- (2,4);
 \end{tikzpicture}

  \\
$M_4 \cong M(K_4)$ & $M_5$ & $M_6 \cong P_7$
 \\
 \\

  \begin{tikzpicture}[scale=0.3]
 
  \filldraw (-4,0) circle (10pt);
  \filldraw (0,0) circle (10pt);
  \filldraw (4,0) circle (10pt);
  
  \filldraw (-2,4) circle (10pt);
  \filldraw (0,4) circle (10pt);
  \filldraw (2,2) circle (10pt);

  \filldraw (0,8) circle (10pt);

  \draw[thick] (-4,0) -- (4,0);
  \draw[thick] (-4,0) -- (0,8);
  \draw[thick] (0,0) -- (0,8);
  \draw[thick] (4,0) -- (0,4);
 \end{tikzpicture}

 &

 \begin{tikzpicture}[scale=0.3]
 
  \filldraw (-4,0) circle (10pt);
  \filldraw (0,0) circle (10pt);
  \filldraw (4,0) circle (10pt);
  
  \filldraw (-2,4) circle (10pt);
  \filldraw (0,4) circle (10pt);
  \filldraw (4,4) circle (10pt);

  \filldraw (0,8) circle (10pt);
  \filldraw (4,8) circle (10pt);

  \draw[thick] (-4,0) -- (4,0);
  \draw[thick] (-4,0) -- (0,8);
  \draw[thick] (0,0) -- (0,8);
  \draw[thick] (4,0) -- (4,8);
 \end{tikzpicture}

 &

  \begin{tikzpicture}[scale=0.3]
 
  \filldraw (-4,0) circle (10pt);
  \filldraw (0,0) circle (10pt);
  \filldraw (4,0) circle (10pt);
  
  \filldraw (-2,4) circle (10pt);
  \filldraw (0,4) circle (10pt);
  \filldraw (4,4) circle (10pt);

  \filldraw (0,8) circle (10pt);
  \filldraw (4,8) circle (10pt);

  \draw[thick] (-4,0) -- (4,0);
  \draw[thick] (-4,0) -- (0,8);
  \draw[thick] (0,0) -- (0,8);
  \draw[thick] (4,0) -- (4,8);
  \draw[thick] (-2,4) -- (4,4);
 \end{tikzpicture}

 \\
 $M_7$ & $M_8$ & $M_9$
 
 \end{tabular}
 \]
 \caption{Matroids $M_1$--$M_7$ are known excluded minors for the class of positroids, and $M_8$ and $M_9$ have the excluded minor from Figure \ref{fig: a rank-3 excluded minor} as a restriction.} 
 \label{fig: rank-3 excluded minors}
\end{figure}

Similarly, we can show that rank-$4$ matroids satisfying a certain structural property are not positroids.

\begin{proposition}\label{prop: rank 4 excluded minor}
    If $M$ is a simple rank-$4$ matroid with a line $L$, distinct points $e_1, e_2, e_3$ on $L$, and distinct planes $P_1$, $P_2$, $P_3$ through $L$ so that for $i = 1,2,3$ the point $e_i$ is on at least two long lines in $P_i$, then $M$ is not a positroid.
    In particular, $M$ has a minor from Figure \ref{fig: rank-3 excluded minors} or Figure \ref{fig: rank-4 excluded minor}.
\end{proposition}
\begin{proof}
    For $i=1,2,3$, let $L_i$ denote the line in $P_i$ other than $L$ which contains $e_i$. Note that by restricting to 3 points on each $L_i$, including the point $e_i$, we may assume that $|M| = 9$ and $E(M) = L_1 \cup L_2 \cup L_3$. If the $3$-element circuits of $M$ are precisely $L, L_1, L_2, L_3$ and the $4$-element circuits are precisely the $4$-element sets contained in $P_i$ for some $i$ and not containing $L$ or $L_i$, then $M$ is the matroid of Figure \ref{fig: rank-4 excluded minor}.
    This matroid appears in work of Bonin, where it is shown to be an excluded minor for the class of positroids \cite[Figure 4]{Bonin-2024}. Thus, we may assume that $M$ has a $3$-element circuit that is not $L$ or any $L_i$, or $M$ has a $4$-element circuit that is not contained in $P_i$ for any $i$.
    We will consider these cases separately.

    First suppose that $M$ has a $3$-element circuit $L'$ that is not $L$ or any $L_i$.
    Note that $|P_i|=|L\cup L_i|=5$, so if $L' \subseteq P_i$ for some $i$, then by the pigeonhole principle, $L'$ contains two points from $L$ or $L_i$.
    Therefore $L$ or $L_i$ spans $P_i$, a contradiction.
    If $L'$ contains exactly two points from $P_i$ for some $i$, then $P_i$ spans a point outside of $P_i$, which contradicts that $P_i$ is a flat of $M$. 
    So $L'$ contains a point $f_i\in P_i -  L$ for $i=1,2,3$. We will contract by the point $g_1$ which is in $L_1$ but neither $L$ nor $L'$. 
    We claim that the sets $L, L_2, L_3, L'$ span distinct long lines of $M/g_1$.
    First, we show that these sets span lines in $M/g_1$. Note that $r_M(L \cup g_1) = 3$ or else $L = L_1$, $r_M(L_i \cup g_1) = 3$ for $i = 2,3$ or else $r_M(P_1 \cup P_i) = 3$, and $r_M(L' \cup g_1) = 3$ or else $\cl(L') = L_1$.
    Therefore each of $L, L_2, L_3, L'$ has rank $2$ in $M/g_1$.
    Now we show that no two of these sets span the same line in $M/g_1$. If they did, then their union has rank $2$ in $M/g_1$, and therefore their union together with $g_1$ has rank $3$ in $M$.
    But it is straightforward to check that in each case the pair of lines together with $g_1$ spans $M$ and so $r(M) = 3$, a contradiction.
    So $\cl_{M/g_1}(L)$ is a long line of $M/g_1$ with three distinct points each on a second long line, and it follows from Proposition \ref{prop: rank 3 excluded minor} that $M/g_1$ has a minor from Figure \ref{fig: rank-3 excluded minors}.

    Now suppose that $M$ has a $4$-element circuit $C$ that is not contained in $P_i$ for $i =1,2,3$.
    Suppose $C$ is contained in $P_i \cup P_j$. 
    Then, since $|C| = 4$ and $|P_i\cup P_j|=|L_i\cup L_j\cup \{e_k\}|=7$ we see that $C$ contains $2$ points from $L_i$ or from $L_j$ by the pigeonhole principle. Without loss of generality, we say $C$ contains two points from $L_i$. Since $C$ must not contain a long line, the remaining $2$ points of $C$ lie in $L_j\cup (L-L_i)$. Since $C$ is not contained in $P_i$, at least one of those points must lie in $L_j$ but not $L-L_i$. If the fourth point lies in $L-L_i$, then $\cl(C)$ contains $L_i$ and $L$, and consequently $L_j$. On the other hand, if the fourth point lies in $L_j$, then $\cl(C)$ contains $L_i$ and $L_j$, and consequently $L$. In either case, $C$ spans the entire matroid, a contradiction.    
    So by symmetry we may assume that $C$ consists of the two elements of $L_1 - L$, one element from $L_2 - L$, and one element from $L_3 - L$.
    Let $g_1 \in C \cap L_1$, and consider $M/g_1$.
    We claim that the sets $L, L_2, L_3, C - g_1$ span distinct long lines of $M/g_1$.
    As argued in the previous paragraph, the sets $L, L_2, L_3$ have rank $2$ in $M/g_1$.
    And $C - g_1$ has rank $2$ in $M/g_1$ because $r_M(C) = 3$.
    So if two of these sets span the same line in $M/g_1$, then their union has rank $2$ in $M/g_1$, and therefore their union together with $g_1$ has rank $3$ in $M$.
    It is straightforward to check that in each case the two sets together with $g_1$ span $M$ and so $r(M) = 3$, a contradiction.
    So $\cl_{M/g_1}(L)$ is a long line of $M/g_1$ with three distinct points each on a second long line, and it follows from Proposition \ref{prop: rank 3 excluded minor} that $M/g_1$ has a minor from Figure \ref{fig: rank-3 excluded minors}.     
\end{proof}

\begin{figure}[htb] 
\[
 \beginpgfgraphicnamed{tikz/fig49}
 \begin{tikzpicture}[x=0.08cm,y=0.07cm]

  \draw ( 0,-5) -- ( 0,55) -- (48,43) -- (48,-17) -- (26, -11.5);
  \draw ( 0,-5) -- ( 0,55) -- (-40,45) -- (-40,-15) --cycle;
  \draw ( 0,-5) -- ( 0,55) -- (26,36) -- (26,-24) --cycle;

  \draw[thick] ( 0, 10) -- ( -32,2);
  \draw[thick] ( 0, 25) -- ( 21,10);
  \draw[thick] ( 0, 10) -- ( 0,40);

  \draw[dashed, thick, color = lightgray] (0, 40) -- (26,33.5);
  \draw[thick] (26, 33.5) -- (40,30);

  \filldraw ( 0,10) circle (3pt);
  \filldraw ( 0,25) circle (3pt);
  \filldraw ( 0,40) circle (3pt);

  \filldraw (40,30) circle (3pt);
  \filldraw (32,32) circle (3pt);

  \filldraw (10.5,17.5) circle (3pt);
  \filldraw (21,10) circle (3pt);

  \filldraw (-16,6) circle (3pt);
  \filldraw (-32,2) circle (3pt);

 \end{tikzpicture}
 \endpgfgraphicnamed
 \]
 \caption{A rank-$4$ excluded minor for the class of positroids.} 
 \label{fig: rank-4 excluded minor}
\end{figure}

\section{The extremal examples} \label{sec: the extremal examples}

In this section we will define the family of positroids for which equality holds in Theorem \ref{thm: main}, and prove that they are in fact positroids and are $U_{2, \ell + 2}$-minor-free.
Given matroids $M$ and $N$ with $E(M) \cap E(N) = \{e\}$, the \emph{parallel connection} of $M$ and $N$ with basepoint $e$ is the matroid $P_e(M, N)$ on ground set $E(M) \cup E(N)$ with the following set as its set of flats: 
$$\{F \subseteq E(M) \cup E(N) \colon F \cap E(M) \textrm{ is a flat of $M$ and } F \cap E(N) \textrm{ is a flat of $N$}\}.$$

 \begin{definition} \label{def: the extremal family}
     For each integer $\ell \ge 2$, let $\cM_{2,\ell}$ be the class of matroids isomorphic to $U_{2,\ell+1}$, and then for each integer $r \ge 3$ let $\cM_{r,\ell}$ be the class of matroids of the form $P_e(M, N)$ with $M \in \cM_{r-1,\ell}$ and $N \cong U_{2, \ell + 1}$.
 \end{definition}

We next illustrate the smallest non-trivial case.

\begin{example}\label{ex:parallel connection}
Let $M$ and $N$ both be isomorphic to the uniform matroid $U_{2,3}$, on ground sets $\{1,2,3\}$ and $\{3,4,5\}$, respectively. We construct $P_3(M,N)\in \mathcal{M}_{3,2}$.
The parallel connection $P_3(M,N)$ is the matroid on $\{1,2,3,4,5\}$ for which:
\begin{itemize}
    \item Singletons $\{i\}$ are flats of rank~$1$.
    \item The sets $\{1,2,3\}$ and $\{3,4,5\}$ are flats of rank~$2$.
    \item Any pair $\{i,j\}$ with $i \in \{1,2\}$ and $j \in \{4,5\}$ is also a flat of rank~$2$, since its intersections with both $M$ and $N$ are rank-$1$ flats.
    \item The full ground set $\{1,2,3,4,5\}$ is a flat of rank~$3$.
\end{itemize}

Geometrically, $P_3(M,N)$ can be viewed as two lines sharing a single common point $3$.  
This configuration is illustrated in \Cref{fig: parallel connection example}.

\begin{figure}[htb] 
\[
 \begin{tikzpicture}[scale=0.3]
 
  \filldraw (-3,0) circle (10pt) node[below=4pt] {$3$};
  \filldraw (0,0) circle (10pt) node[below=4pt] {$4$};
  \filldraw (3,0) circle (10pt) node[below=4pt] {$5$};

  \filldraw (-3,3) circle (10pt) node[left=4pt] {$2$};
  \filldraw (-3,6) circle (10pt) node[left=4pt] {$1$};
  
   \draw[thick] (-3,6) -- (-3,0);
  \draw[thick] (-3,0) -- (3,0);
  
 \end{tikzpicture}
 \]
 \caption{The parallel connection of two matroids $M$ and $N$, each isomorphic to $U_{2,3}$.} 
 \label{fig: parallel connection example}
\end{figure}

\end{example}

Since each subsequent parallel connection with $U_{2,\ell+1}$ increases the rank by one \cite[Proposition 7.1.15(i)]{Oxley-2011}, each matroid in $\cM_{r,\ell}$ has rank $r$.
And since each subsequent parallel connection increases the number of elements by $\ell$, each matroid in $\cM_{r,\ell}$ has $\ell(r - 1) + 1$ elements.
Bonin proved that the class of positroids is closed under parallel connections \cite[Corollary 4.18]{Bonin-2024} and, by Example~\ref{ex:uniform}, rank-$2$ uniform matroids are positroids. Thus, every matroid in $\cM_{r,\ell}$ is a positroid.
The following lemma implies that each matroid in $\cM_{r,\ell}$ is $U_{2,\ell+2}$-minor-free.

\begin{lemma} \label{lem: the extremal family is U2l-minor-free}
    For each integer $\ell \ge 1$, if a matroid $M$ is $U_{2, \ell+2}$-minor-free, then $P_e(M, U_{2, \ell+1})$ is $U_{2,\ell+2}$-minor-free.
\end{lemma}
\begin{proof}
    Let $N = P_e(M, U_{2, \ell+1})$, and suppose that $N/X \del Y \cong U_{2, \ell+2}$. 
    We may assume that $r(N/X) = 2$.
    Since $U_{2,\ell+2}$ is loopless we may further assume that $X$ is a flat of $N$.
    If $e \in X$, then by \cite[Proposition 7.1.15(iii)]{Oxley-2011}, $U_{2,\ell+2}$ is a minor of the direct sum $N/e=(M/e) \oplus (U_{2, \ell+1}/e)$.
    Since $U_{2,\ell+2}$ is connected it follows that $M/e$ has $U_{2,\ell+2}$ as a minor, a contradiction.
    So $e \notin X$.
    By \cite[Proposition 7.1.15(v)]{Oxley-2011}, if we write $X_1 = X \cap E(M)$ and $X_2 = X - E(M)$, then $U_{2, \ell+2}$ is a minor of $N/X=P_e(M/X_1, U_{2, \ell+1}/X_2)$. This also tells us that $r(M/X_1)+r(U_{2, \ell+1}/X_2)=r(N/X)+1=3$.
    If $X_2$ is empty, then $r(M/X_1)=1$ and so each element in $E(M) - X_1$ is parallel in $P_e(M/X_1, U_{2, \ell+1}/X_2)$ to $e$.
    Then $P_e(M/X_1, U_{2, \ell+1}/X_2)$ simplifies to $U_{2,\ell+1}$ and therefore has no $U_{2, \ell+2}$-minor.
    So $X_2$ is non-empty.
    Then $P_e(M/X_1, U_{2, \ell+1}/X_2)$ simplifies to $\si(M/X_1)$, and therefore has no $U_{2, \ell+2}$-minor.
\end{proof}

So every matroid in $\cM_{r,\ell}$ is a simple, rank-$r$ positroid with no $U_{2,\ell+2}$-minor and $\ell(r - 1) + 1$ elements, and is therefore a sharp example for Theorem \ref{thm: main}.

\section{The main proof} \label{sec: the main proof}

We can now prove our main result, Theorem \ref{thm: main}.
We will in fact show the following more general result, which directly implies Theorem \ref{thm: main} via Propositions \ref{prop: rank 3 excluded minor} and \ref{prop: rank 4 excluded minor}.

\begin{theorem} \label{thm: main result with proof}
For all integers $r, \ell \ge 1$, if $M$ is a simple rank-$r$ matroid with no minor from Figure \ref{fig: rank-3 excluded minors} or Figure \ref{fig: rank-4 excluded minor} and no $U_{2,\ell+2}$-minor, then $|M| \le \ell(r - 1) + 1$.
Moreover, equality holds if and only if $M$ can be obtained by taking parallel connections of copies of $U_{2, \ell+1}$.
\end{theorem}
\begin{proof}
First suppose that $\ell = 1$.
If $M$ is a simple rank-$r$ matroid with more than $r$ elements, then $r \ge 2$ and $M$ has a circuit with at least three elements.
But every circuit with at least three elements has a $U_{2,3}$-minor \cite[Example 3.1.9]{Oxley-2011}, a contradiction.
So $M$ is a rank-$r$ matroid with $r$ elements, and it follows that $M$ is the parallel connection of $r - 1$ copies of $U_{2,2}$, as desired.
So we may assume that $\ell \ge 2$.
We will proceed by induction on $r$.
The statement is clearly true when $r \le 2$ so we may assume that $r \ge 3$.
If $M$ has an element $e$ that is not on any long lines, then $M/e$ is simple, $|M/e| = |M| - 1$, and $r(M/e) = r - 1$, so by induction we have
\begin{align*}
    |M| = 1 + |M/e| \le 1 + \ell(r - 2) + 1 = \ell(r - 1) + 2 - \ell < \ell(r - 1) + 1,
\end{align*}
as desired.
So we may assume that every element of $M$ is on a long line.
Let $L$ be a long line of $M$.
We will now consider the connected components of $M/L$.
Let $\cX$ be the partition of $E(M) - L$ given by the ground sets of the connected components of $M/L$.
Let $X \in \cX$, and let $N = M|(L \cup X)$.
We will first bound the number of elements of $N \del L$.

\begin{claim} \label{claim: bound the size of each component}
    At most two elements in $L$ are on a long line of $N$ other than $L$, and $|N \del L| \le \ell(r(N) - 2)$.
\end{claim}
\begin{proof}
Suppose that distinct elements $e_1$, $e_2$, and $e_3$ in $L$ are on long lines $L_1$, $L_2$, and $L_3$, respectively, of $N$ other than $L$.
Let $K = \si(N/L)$.
By deleting elements from nontrivial parallel classes of $N/L$ we may assume that $K$ is a restriction of $N/L$.
Note that the elements of $K$ are in one-to-one correspondence with the planes of $N$ containing $L$.
For $i = 1,2,3$ let $p_i$ be the element of $K$ in the plane $\cl_N(L \cup L_i)$ of $N$.
Since $K$ is connected, it follows from \cite[Theorem 4.3.1]{Oxley-2011} that $K$ has a connected minor $K_1 = K \del D/C$ on ground set $\{p_1, p_2, p_3\}$.
We may assume that $C$ is independent in $K$ and therefore in $N/L$, which implies that $r_{N/C}(L) = r_N(L) = 2$.
Since $K_1$ is connected, it is isomorphic to $U_{1,3}$ or $U_{2,3}$.
Let $N_1$ be obtained from $N$ by deleting, for each element $d$ in $D$, the elements in $\cl_N(L \cup d) - L$, and then contracting $C$.
Then $K_1 = \si(N_1/L)$ and $r(N_1) = r(K_1) + 2$.
We will show that $N_1$ has a minor from Figure \ref{fig: rank-3 excluded minors} or \ref{fig: rank-4 excluded minor}.

Since $r_{N_1}(L) = r_{N/C}(L) = 2$ and $K_1$ has ground set $\{p_1, p_2, p_3\}$, for each $i \in \{1,2,3\}$ the set $\cl_{N_1}(L \cup p_i)$ is a plane of $N_1$ containing $L$.
Then $r_{N_1}(L \cup p_i) = r_N(L \cup p_i)$, which implies that $r_{N_1}(L_i) = r_N(L_i) = 2$ for $i = 1,2,3$.
If $K_1 \cong U_{1,3}$, then $N_1$ is a rank-$3$ matroid with a long line $\cl_{N_1}(L)$ and long lines $\cl_{N_1}(L_i)$ for $i = 1,2,3$ through distinct elements of $L$.
By Proposition \ref{prop: rank 3 excluded minor}, $M$ has a minor from Figure \ref{fig: rank-3 excluded minors}, a contradiction.
So $K_1 \cong U_{2,3}$.
Then $N_1$ is a rank-$4$ matroid with three distinct planes through $\cl_{N_1}(L)$ (namely $\cl_{N_1}(L \cup L_i)$ for $i = 1,2,3$) so that each has a long line through a distinct element of $\cl_{N_1}(L)$.
By Proposition \ref{prop: rank 4 excluded minor}, $M$ has a minor from Figure \ref{fig: rank-3 excluded minors} or \ref{fig: rank-4 excluded minor}, a contradiction.

So there is some $f \in L$ on no long lines of $N$ other than $L$.
Then $|N| - \elem(N/f) = |L| - 1$, so 
by induction on $r$ we have
\begin{align*}
    |N| = |L| - 1 + \elem(N/f) \le |L| - 1 + \ell(r(N/f) - 1) + 1 
    = |L| + \ell(r(N) - 2),
\end{align*}
and the claim holds.
\end{proof}

We can now show that $|M| \le \ell(r - 1) + 1$.
Let $\cX = \{X_1, X_2, \dots, X_k\}$ for some positive integer $k$, and for each $i \in [k]$ let $N_i = M|(L \cup X_i)$.
Since the sets in $\cX$ are the ground sets of the connected components of $M/L$ we have $\sum_{i = 1}^k r(N_i/L) = r(M/L)$, so $\sum_{i = 1}^k \left(r(N_i) - 2\right) = r - 2$.
Using Claim \ref{claim: bound the size of each component} we have
\begin{align*}
    |M\del L| &= \sum_{i = 1}^k |N_i \del L| \\
            &\le \sum_{i = 1}^k \ell(r(N_i) - 2) \\
            &= \ell \sum_{i = 1}^k \left(r(N_i) - 2 \right)\\
            &= \ell(r - 2).
\end{align*}
Therefore 
\begin{align*}
    |M| \le \ell(r - 2) + |L| \le \ell(r - 2) + \ell + 1 = \ell(r - 1) + 1,
\end{align*} as desired.

It remains to characterize when equality holds.
Suppose that $r \ge 2$ and $|M| = \ell(r - 1) + 1$.
From the inequalities above and the fact that $L$ is an arbitrary long line, we see that every long line of $M$ has $\ell + 1$ elements.
In particular, if $r = 2$ then $M \cong U_{2,\ell+1}$, as desired.
So we may assume that $r \ge 3$.
We first show that $M$ is connected.
If not, then $M = M_1 \oplus M_2$, and by induction on $r$ we have 
\begin{align*}
    |M| &= |M_1| + |M_2| \\
    &\le \ell(r(M_1) - 1) + 1 + \ell(r(M_2) - 1) + 1 \\
    &= \ell(r - 1) + 2 - \ell <
    \ell(r - 1) + 1,
\end{align*}
a contradiction, where we use the fact that we considered $\ell=1$ earlier in the final inequality.
So, $M$ is connected.

We next identify a special element of $L$.

\begin{claim}
    Some $x \in L$ is on more than one long line of $M$.
\end{claim}
\begin{proof}
Let $e \in L$ and let $M_1 = M/e$.
Let $P = L - e$, and note that $P$ is a nontrivial parallel class of $M_1$. 
If $e$ is on more than one long line of $M$, we have nothing to show. 
Otherwise, $\elem(M_1) = |M| - \ell = \ell(r(M_1) - 1) + 1$, so by induction, $\si(M_1)$ is obtained by taking parallel connections of copies of $U_{2,\ell+1}$.
This implies that $P$ is on a long line $L_1$ of $M_1$.
Let $F$ be the plane $L \cup L_1$ of $M$.
Note that $|F| \ge |L| + |L_1| - 1 \ge 2(\ell + 1) - 1 = 2\ell + 1$.
If $F - L$ has a $3$-element independent set $I$, then $M|F$ has an $M(K_4)$-restriction (by restricting to $I\cup L$ if each pair of points in $I$ spans a point on $L$) or a $U_{2, \ell+2}$-minor, a contradiction.
So $F - L$ has rank $2$.
If $\cl(F - L)$ does not intersect $L$, then for every $f \in F - L$ we have $(M|F)/f \cong U_{2,\ell+2}$, a contradiction. 
So $\cl(F - L)$ intersects $L$ in some element $x$.
Since $|F - L| \ge \ell \ge 2$, it follows that $\cl(F - L)$ is a long line of $M$ that contains $x$ and is not equal to $L$, as desired.
\end{proof}

Let $L'$ be a long line of $M$ through $x$ other than $L$.
First suppose that $M/x$ is connected.
By \cite[Proposition 4.1.3]{Oxley-2011} there is a circuit $C$ of $\si(M/x)$ that intersects $L$ and $L'$.
Let $C_1 \subseteq C$ so that $|C - C_1| = 3$ and $C_1$ is disjoint from $L$ and $L'$.
Then $\si(M/x/C_1)$ has a $3$-element circuit $C - C_1$ that intersects $L$ and $L'$.
Note that $C_1$ is independent in $M/L$ and $M/L'$, so $L$ and $L'$ both have rank $2$ in $M/C_1$.
Let $M_1$ be the restriction of $M/C_1$ to $\cl_{M/C_1}(L \cup L')$.
Then $M_1$ is a rank-$3$ matroid. 
Since $\si(M/x/C_1)$ has $\si(M_1/x)$ as a restriction, $M_1/x$ has a $3$-element circuit $C - C_1$, and so there is some $c \in C - C_1$ that is in $E(M_1)$ but not $\cl_{M_1}(L)$ or $\cl_{M_1}(L')$.
If $c$ is on two or more long lines of $M_1$, each of which has nonempty intersection with both $\cl_{M_1}(L)$ and $\cl_{M_1}(L')$, then $M_1$ has an $M(K_4)$-restriction, a contradiction.
Otherwise, $M_1/c$ has a $U_{2, \ell+2}$-restriction, a contradiction.
So $M/x$ is disconnected.

Let $M/x = N_1 \oplus N_2$ for non-empty matroids $N_1$ and $N_2$.
For $i = 1,2$ let $M_i = M|(E(N_i) \cup x)$.
By \cite[Theorem 7.1.16(ii)]{Oxley-2011} we know that $M = P_x(M_1, M_2)$.
If $|M_1| < \ell(r(M_1) - 1) + 1$, then
\begin{align*}
|M| &= |M_1| + |M_2| - 1 \\
&< \ell(r(M_1) - 1) + 1 + \ell(r(M_2) - 1) + 1 -1\\
& = \ell(r(M_1) + r(M_2) - 2) + 1 \\
&= \ell(r(M) - 1) + 1,
\end{align*}
a contradiction.
So, by induction on $r$, $M_1$ and $M_2$ are both obtained by taking parallel connections of copies of $U_{2,\ell+1}$.
Since $M = P_x(M_1, M_2)$ and parallel connection is associative (this follows from \cite[Proposition 7.1.23]{Oxley-2011}), it follows that $M$ can also be obtained by taking parallel connections of copies of $U_{2, \ell+1}$.
\end{proof}

When $\ell \ge 5$, the class of matroids considered in Theorem \ref{thm: main result with proof} contains the matroid of Figure \ref{fig: a rank-3 excluded minor}, and therefore strictly contains the class of $U_{2,\ell+1}$-minor-free positroids.
Moreover, as $\ell$ grows, the class contains an increasing number of rank-$3$ excluded minors for the class of positroids, as shown in \cite[Corollary 4.12]{Blum-2001}.
When $\ell = 1$ the matroids of Theorem \ref{thm: main result with proof} are all uniform and are therefore positroids, and when $\ell = 2$ the matroids of Theorem \ref{thm: main result with proof} are precisely the binary matroids with no $M(K_4)$-minor, which are known to be positroids by \cite[Theorem 5.1]{Blum-2001}.
However, when $\ell \in \{3,4\}$ it is unclear whether or not the matroids of Theorem \ref{thm: main result with proof} are all positroids.
We leave this as an open problem.

\section{Other minor-closed classes of matroids}\label{sec:minor-closed-classes}

In this section, we define gammoids and study subclasses of them, commenting on their relations to the class of positroids and proving an appropriate analogue of \Cref{thm: main result with proof} for each. Specifically, we look at laminar, colaminar, lattice path, multi-path, bicircular and path-circular matroids. A summary of this section, showing the classes we consider, the containments between them, and whether \Cref{thm: main result with proof} holds for each class, can be found in \Cref{fig:poset-of-classes}.

Gammoids are most easily defined in terms of transversal matroids. 
A matroid $M$ on ground set $E$ is \emph{transversal} if there is a collection $\cA$ of subsets of $E$ so that the independent sets of $M$ are precisely the \emph{partial transversals} of $\cA$, where $I \subseteq E$ is a partial transversal of $\cA$ if there is a subset $\cA'$ of $\cA$ and bijection $\psi \colon \cA' \to I$ such that $\psi(A) \in A$ for all $A \in \cA'$.
Equivalently, if $G$ is the bipartite graph with bipartition $(E, \cA)$ and $e \in E$ adjacent to $A \in \cA$ if and only if $e \in A$, then $I \subseteq E$ is a partial transversal of $\cA$ if and only if there is a matching that saturates $I$. For such a transversal matroid, we write $M=M(E,\mathcal{A})$ and we say that $\cA$ is a \emph{presentation} of $M$.

The class of transversal matroids is closed under restriction but not duality or contraction.
\emph{Strict gammoids} are duals of transversal matroids, and \emph{gammoids} are minors of transversal matroids.
The classes of transversal matroids and strict gammoids are incomparable with the class of positroids, but every positroid is a gammoid \cite{Chidiac-Hochstattler-2024}.

We will use a characterization of strict gammoids due to Mason.
Let $E$ be the ground set of a matroid $M$, and let $\alpha_M \colon 2^E \to \mathbb Z$ be defined recursively by $\alpha_M(\emptyset) = 0$ and $\alpha_M(X) = |X| - r(X) - \sum_{F \in \overline{\cF}(X)}\alpha_M(F)$, where $\overline{\cF}(X)$ is the set of proper subsets of $X$ that are also flats of $M$.
A result of Mason \cite[Theorems 2.2 and 2.4]{Mason-1972} states that:

\begin{theorem}[Mason \cite{Mason-1972}] \label{thm: Mason result}
A matroid $M$ is a strict gammoid if and only if $\alpha_M(X) \ge 0$ for each $X \subseteq E$.   
\end{theorem}

Gammoids satisfy further nice properties which, in light of \Cref{sec: introduction}, make them reasonable candidates for a version of \Cref{thm: main}.
Specifically, $M(K_4)$ is not a gammoid \cite[page 435, Exercise 11]{Oxley-2011}, every gammoid is $\bR$-representable \cite[Proposition 3.9]{Ingleton-Piff-1973}, and the class of gammoids is minor-closed by definition. 
While we leave this problem open, we discuss it further in \Cref{subsec:MK4-free}. 
Here, we will conclusively resolve the analogous problem for a number of established minor-closed subclasses of gammoids.

\subsection{Laminar matroids} \label{sec: laminar}

Laminar matroids were first studied in optimization in the context of the matroid secretary problem \cite{Secretary-laminar-2011, Secretary-laminar-2013}, and have since been studied solely for their structural properties \cite{laminar-matroids, Generalized-laminar}.
They can be realized by a capacity function in the following way: Let $\mathcal{A}$ be a \emph{laminar family} of sets, that is a family of sets in which, for any $A,B\in\mathcal{A}$, either $A\cap B=\emptyset$, $A\subseteq B$, or $B\subseteq A$. Let $E$ be a finite set and let $c:E\rightarrow\mathbb{R}$ be a function, which we call the \emph{capacity}. We call a subset $I\subseteq E$ independent if $c(I\cap A)\geq |I\cap A|$ for all $A\in \mathcal{A}$. This defines a matroid $M(E, c, \mathcal{A})$. 
Any matroid isomorphic to such a matroid is called \emph{laminar}. 
The class of laminar matroids is closed under minors but not duality \cite[Theorem 1.2]{laminar-matroids}; a matroid dual to a laminar matroid is called \emph{colaminar}. 
As described in \cite{laminar-matroids}, Finkelstein \cite{Finkelstein_thesis} proved that all laminar and colaminar matroids are gammoids.
However, the class of laminar matroids is incomparable with the class of transversal matroids, and there are transversal matroids that are not colaminar. 
The matroid in Figure \ref{fig: parallel connection example} (known as $Y_3$) and its dual are both transversal, but $Y_3$ is not laminar and $Y_3^*$ is not colaminar \cite[Theorem 1.2]{laminar-matroids}, and the matroid obtained from $U_{2,3}$ by making a parallel copy of each element is laminar but not transversal \cite[Figure 1.20]{Oxley-2011}.
On the other hand, we will prove that every colaminar matroid is transversal.
To the best of our knowledge, this was not previously known.

\begin{proposition} \label{prop: colaminar implies transversal}
Every laminar matroid is a strict gammoid. Dually, every colaminar matroid is transversal.
\end{proposition}
\begin{proof}
By \cite[Theorem 1.5]{laminar-matroids}, every laminar matroid can be obtained from copies of $U_{1,1}$ by taking direct sums and truncations of previously constructed matroids.
So it suffices to show that the class of strict gammoids is closed under direct sums and truncations.
By \cite[Proposition 4.2.11]{Oxley-2011} the class of transversal matroids is closed under direct sums, so by \cite[4.2.21]{Oxley-2011} and duality the class of strict gammoids is closed under direct sums.

For truncations, we will use Theorem \ref{thm: Mason result}.
Let $N$ be a strict gammoid on ground set $E$ and let $N'$ be the truncation of $N$.
For a set $X \subseteq E$, every proper subset of $X$ that is a flat of $N'$ is also a flat of $N$, and a every proper subset of $X$ that is a flat of $N$ is either a flat of $N'$ (if it has rank at most $r(N) - 2$) or a hyperplane of $N$.
In other words, writing $\overline \cH_N(X)$ for the set of hyperplanes of $N$ properly contained in $X$, we have $\overline \cF_{N}(X) = \overline \cF_{N'}(X) \cup \overline \cH_N(X)$.
Therefore 
\begin{align}
\alpha_{N}(X) &= |X| - r_{N}(X) -  \sum_{F \in \overline{\cF}_{N}(X)}\alpha_N(F)\\
&= |X| - r_{N}(X) -  \sum_{F \in \overline{\cF}_{N'}(X)}\alpha_N(F)
-  \sum_{F \in \overline{\cH}_{N}(X)}\alpha_N(F) \\
&= |X| - r_{N}(X) -  \sum_{F \in \overline{\cF}_{N'}(X)}\alpha_{N'}(F)
-  \sum_{F \in \overline{\cH}_{N}(X)}\alpha_N(F) \\
&\le |X| - r_{N'}(X) - \sum_{F \in \overline{\cF}_{N'}(X)}\alpha_{N'}(F)\\
&= \alpha_{N'}(X).
\end{align}
Line (3) holds because if $F$ is a proper flat of $N'$, then $\alpha_N(F) = \alpha_{N'}(F)$ because $N|F = N'|F$, and $N$ and $N'$ have the same flats contained in $F$.
Line (4) holds because $r_N(X) \ge r_{N'}(X)$, and $\alpha_N(F) \ge 0$ for all $F \in \overline \cH_N(X)$ because $N$ is a strict gammoid.
Since $\alpha_N(X) \ge 0$ for all $X \subseteq E$ we see that $\alpha_{N'}(X) \ge 0$ for all $X \subseteq E$, so $N'$ is a strict gammoid.
\end{proof}

There is an equivalent characterization of laminar matroids in more familiar matroid-theoretic language due to Fife and Oxley \cite[Theorem 1.1]{laminar-matroids}.

\begin{theorem}[Fife, Oxley \cite{laminar-matroids}] \label{thm: laminar circuits}
A matroid $M$ is laminar if and only if for any two disjoint circuits $C_1$ and $C_2$ of $M$, either $\cl(C_1)\subseteq \cl(C_2)$ or $\cl(C_2)\subseteq \cl(C_1)$.
\end{theorem}

We prove another nice property of laminar matroids, which will easily imply that every laminar matroid is a positroid.
A flat $F$ of a matroid $M$ is \emph{connected} if $M|F$ is connected.

\begin{proposition} \label{prop: orderings of laminar matroids}
If $M$ is a laminar matroid on ground set $E$, then there is a total ordering $<$ of $E$ such that every connected flat $F$ is an interval in $<$ order.
\end{proposition}
\begin{proof}
We proceed by induction on $|E|$.
The statement is clearly true when $|E| = 1$, so we may assume that $|E| \ge 2$.
If $M$ has a loop then every connected flat has at most one element and the statement holds trivially, so we may assume that $M$ is loopless.
Let $\cF = \{F_1, F_2, \dots, F_t\}$ be the collection of maximal proper connected flats of $M$.
Since $M$ is loopless, every element is in a set in $\cF$.
By \cite[Corollary 2.14]{laminar-matroids}, each pair of sets in $\cF$ is either disjoint or nested.
Since each set in $\cF$ is maximal, it follows that the sets in $\cF$ are pairwise disjoint.
For each $F \in \cF$, the matroid $M|F$ is laminar because the class of laminar matroids is minor-closed \cite[Lemma 3.1]{laminar-matroids}.
By induction, for each $i \in [t]$ there is a total ordering $<_i$ of $M|F_i$ such that every connected flat is an interval in $<_i$ order.
Let $<$ be the concatenation of $<_1, <_2, \dots, <_t$ in that order.
Then $<$ is a total order of $E$, and since every connected flat of $M$ is a connected flat of $M|F_i$ for some $i \in [t]$ by the maximality of the flats in $\cF$, every connected flat of $M$ is an interval in $<$ order.
\end{proof}

To the best of our knowledge, the following corollary was not previously known.

\begin{corollary} \label{cor: laminar implies positroid}
Every laminar or colaminar matroid is a positroid.
\end{corollary}
\begin{proof}
By Proposition \ref{prop: orderings of laminar matroids} and Theorem \ref{thm:Bonin's condition}, every laminar matroid is a positroid.
Since the class of positroids is closed under duality, every colaminar matroid is a positroid.
\end{proof}

By Corollary \ref{cor: laminar implies positroid}, the upper bound of Theorem \ref{thm: main result with proof} applies to laminar and colaminar matroids.
Interestingly, while colaminar matroids achieve the bound, laminar matroids do not.
Clearly when $r \ge 3$ and $\ell \ge 2$ every matroid in $\cM_{r,\ell}$ has $Y_3$ as a minor, so no matroid in $\cM_{r,\ell}$ is laminar.
Therefore Theorem \ref{thm: main result with proof} does not determine the maximum number of elements of a simple rank-$r$ laminar matroid with no $U_{2,\ell+2}$-minor. 

\begin{corollary}
    For all integers $r\geq 3$ and $ \ell \ge 2$, if $M$ is a rank-$r$ laminar matroid with no $U_{2,\ell+2}$-minor, then $|M| < \ell(r - 1) + 1$.
\end{corollary}

We leave a tight upper bound as an open problem.
On the other hand, at least one matroid in $\cM_{r,\ell}$ is colaminar for all $r\ge 3$ and $\ell \ge 2$, which gives the following result.

\begin{theorem}
For all integers $r,\ell\geq 1$, the maximum number of elements of a simple rank-$r$ colaminar matroid with no $U_{2,\ell+2}$-minor is $\ell(r - 1) + 1$.
\end{theorem}
\begin{proof}
By Theorem \ref{thm: main result with proof} and Corollary \ref{cor: laminar implies positroid} it suffices to show that $\cM_{r,\ell}$ contains a colaminar matroid.
This is true when $r \le 2$ or $\ell = 1$ because every uniform matroid is laminar, and therefore colaminar, by \cite[Theorem 1.5]{laminar-matroids}.
So we may assume that $r \ge 3$ and $\ell \ge 2$.
For all $r \ge 3$ and $\ell \ge 2$ let $S_r$ be the unique member of $\cM_{r,\ell}$ in which the same element is chosen as the basepoint of each parallel connection.
We claim that $S_r$ is colaminar.
Let $(x, L_1, L_2, \dots, L_{r-1})$ be the partition of $E(S_r)$ so that for each $i \in [r-1]$ the set $L_i \cup x$ is a long line through $x$.
The hyperplanes of $S_r$ are the unions of $r - 2$ long lines through $x$ and the transversals of $(L_1, L_2, \dots, L_{r-1})$.
Therefore the circuits of $S_r^*$ are the sets $L_i$ for $i \in [r-1]$ and the complements of the transversals of $(L_1, L_2, \dots, L_{r-1})$.
The former sets are pairwise disjoint and the latter sets are spanning circuits of $S_r^*$ by \cite[Proposition 2.1.6(i)]{Oxley-2011} since each transversal of $(L_1, L_2, \dots, L_{r-1})$ is independent in $S_r$.
Therefore for any two circuits $C_1, C_2$ of $S_r^*$ with $C_1 \cap C_2\ne \emptyset$, either $C_1$ spans $C_2$ or $C_2$ spans $C_1$, so $S_r^*$ is laminar by Theorem \ref{thm: laminar circuits}.
\end{proof}

\subsection{Minor-closed classes of transversal matroids} \label{subsec:minor-close-transversal}

The rest of the classes we consider are all transversal matroids.
We begin by deriving some useful facts about transversal matroids. We can easily show that the matroids of Figure \ref{fig: rank-3 excluded minors} are not transversal.

\begin{proposition} \label{prop: rank-3 non-transversal matroids}
Every rank-$3$ transversal matroid has at most three long lines.
\end{proposition}
\begin{proof}
Suppose that $M=M(E,\mathcal{A})$ is a rank-$3$ transversal matroid with distinct long lines $L_1,L_2,L_3,L_4$.
By \cite[Lemma 2.4.1]{Oxley-2011} we may assume that $|\cA| = r(M) = 3$; let $\cA = \{A_1, A_2, A_3\}$.
By \cite[page 98, Exercise 2]{Oxley-2011}, each $L_i$ is disjoint from a set in $\cA$.
So there are distinct $i,i' \in [4]$ and some $j \in [3]$ so that $L_i$ and $L_{i'}$ are both disjoint from $A_j$.
But then $r(L_i \cup L_{i'}) \le 2$, a contradiction.
\end{proof}

Using similar reasoning, \Cref{prop: a rank-4 non-transversal matroid} shows that the matroid of Figure \ref{fig: rank-4 excluded minor} is not transversal.
Two sets $X,Y$ in a matroid are \emph{skew} if $r(M) + r(Y) = r(X \cup Y)$.

\begin{proposition} \label{prop: a rank-4 non-transversal matroid}
Every rank-$4$ matroid with a set of three pairwise skew long lines is not transversal.
\end{proposition}
\begin{proof}
Suppose that $M=M(E,\mathcal{A})$ is a rank-$4$ transversal matroid with pairwise skew long lines $L_1,L_2,L_3$.
By \cite[Lemma 2.4.1]{Oxley-2011} we may assume that $|\cA| = 4$; let $\cA = \{A_1, A_2, A_3, A_4\}$.
By \cite[page 98, Exercise 2]{Oxley-2011}, each $L_i$ is disjoint from two sets in $\cA$.
So there are some $i,i' \in [3]$ and some $j \in [4]$ so that $L_i$ and $L_{i'}$ are both disjoint from $A_j$.
But then $r(L_i \cup L_{i'}) \le 3$, a contradiction.
\end{proof}

We now have the following corollary of Theorem \ref{thm: main result with proof}.

\begin{theorem} \label{thm: classes of transversal matroids}
Let $\cM$ be a minor-closed class of transversal matroids.
Then for all positive integers $r$ and $\ell$, every simple rank-$r$ matroid in $\cM$ with no $U_{2, \ell+2}$-minor has at most $\ell(r - 1) + 1$ elements. Moreover, equality only holds for iterated parallel connections of copies of $U_{2,\ell+1}$, if such a matroid is in $\cM$. 
\end{theorem}

There exist many positroids which are not transversal, and thus positroids do not belong to any minor-closed class of transversal matroids. 
See \cite{Positroids-and-Transversals} for a detailed discussion of which positroids are transversal.
For each minor-closed class of transversal matroids we discuss, we determine whether they are subclasses of positroids and whether they achieve the equality of \Cref{thm: classes of transversal matroids}.

\subsubsection{Transversal positroids}

We next study lattice path matroids and multi-path matroids, which are subclasses of both positroids and transversal matroids. 

We begin with lattice path matroids, first defined in \cite{BdMN}.  
A \emph{lattice path} is a path in the integer lattice $\mathbb{Z}^2$ which only takes steps in the directions $(1,0)$ (right-steps) and $(0,1)$ (up-steps).  
Fix two lattice paths $P_1,P_2$ from $(0,0)$ to $(n-r,r)$ for some $1\leq r\leq n$.
Note that such paths contain exactly $n$ steps, which we number in order. 
For instance, step 1 is either an up-step from $(0,0)$ to $(0,1)$ or a right-step $(0,0)$ to $(1,0)$. 
Assume $P_1$ lies weakly above $P_2$, that is, for any $(x,y_1)$ in $P_1$ and $(x,y_2)$ in $P_2$, we have $y_1\geq y_2$.
Define $B(P_1,P_2)$ to be the set containing the set of up-steps of each lattice path that is weakly below $P_1$ and weakly above $P_2$. 
Let $M(P_1,P_2)$ be the matroid with bases $B(P_1,P_2)$.
A \emph{lattice path matroid} is any matroid isomorphic to $M(P_1,P_2)$ for some lattice paths $P_1, P_2$ with $P_1$ lying weakly above $P_2$.

Lattice path matroids are positroids \cite[Lemma 23]{Oh}, and the class of lattice path matroids is closed under minors and duality \cite[Theorem 3.1]{Bonin-deMier-LPMs}.
Lattice path matroids are also transversal, as they are precisely the matroids which can be realized as $M(E,\mathcal{A})$ for $\mathcal{A}$ an antichain of intervals \cite[Theorem 3.3]{BdMN}. 
Sitting in the intersection of the highly structured sets of positroids and of transversal matroids, while also admitting a concrete description in terms of lattice paths, lattice path matroids are a natural testing ground for conjectures and open problems. 
Indeed, lattice path matroids have been studied from a number of perspectives, including categorical \cite{LPM-quotient}, polytopal \cite{LPM-Ehrhart}, and matroid theoretic \cite{BdMN}. 
There are matroids achieving equality in \Cref{thm: classes of transversal matroids} which are lattice path matroids \cite[end of Section 6]{Bonin-deMier-LPMs}.

Multi-path matroids, first introduced in \cite{Multipath}, are a generalization of lattice path matroids. 
They are a larger class dual-closed of matroids which also lies in the intersection of positroids and transversal matroids, and also admit a concrete description in terms of paths. 
This makes them another approachable class to study. 
The description in terms of paths is involved and will not help us here; we instead focus on the presentation of multi-path matroids as transversal matroids. 
\emph{Multi-path matroids} are precisely the matroids which can be realized as $M(E,\mathcal{A})$ on the ground set $E=\{1,\ldots, n\}$ with $\mathcal{A}$ an antichain of intervals in a fixed shifted linear order, that is, an order $i<i+1<\cdots < n<1<\cdots <i-1$ for some choice of $i$. 
All multi-path matroids are positroids \cite[Section 2.6]{Bonin-2024}.
Also, the classes of multi-path matroids and lattice path matroids are incomparable with the classes of laminar matroids and colaminar matroids: $Y_3$ is lattice path by \cite[Theorem 3.1]{Lattice-path-minors2010} but not laminar, $Y_3^*$ is lattice path but not colaminar, and since the class of multi-path matroids is dual-closed and there is a laminar matroid that is not transversal (as discussed in Section \ref{sec: laminar}) and therefore not multi-path, there is also a colaminar matroid that is not multi-path.
Since lattice path matroids are multi-path matroids, there are multi-path matroids achieving equality in \Cref{thm: classes of transversal matroids}. We summarize these observation as follows: 

\begin{corollary}
    For all integers $r,\ell\geq 1$, the maximum number of elements of a simple rank-$r$ lattice path matroid or multi-path matroid with no $U_{2,\ell+2}$-minor is $\ell(r - 1) + 1$.
\end{corollary}

\subsubsection{Non-positroids}

Finally, we consider the classes of bicircular and path-circular matroids, which are transversal but not necessarily positroids.
Given a graph $G$ with edge set $E$, we obtain a \emph{bicircular matroid} $B(G)$ whose independent sets are collections of edges for which each component contains at most a single cycle. 
Originally defined by Sim\~oes Pereira \cite{Bicircular-1972}, bicircular matroids have since appeared in many important conjectures in structural matroid theory (see \cite[Conjecture 6.1]{Exclude-uniform-matroid} or \cite[Conjecture 9.2]{Structure-theory}), and arise as a special case of the frame matroid of a biased graph \cite{Zaslavsky1991}.

Bicircular matroids are transversal. In fact, they can be characterized as precisely those transversal matroids $M(E, \mathcal{A})$ for which each $e\in E$ is contained in at most two sets of $\mathcal{A}$ \cite[Theorem 3.1]{Matthews-Bicircular}.

 The class of bicircular matroids is incomparable with positroids. By \cite[Theorem 4]{bicircular}, $M_3^*$, the dual of the matroid $M_3$ of \Cref{fig: rank-3 excluded minors}, is bicircular. 
 Since $M_3$ is not a positroid and positroids are closed under duality, $M_3^*$ is not a positroid.

\begin{figure}
    \centering
    \scalebox{0.7}{
    \begin{tikzpicture}[
    every node/.style={draw, rectangle, rounded corners, minimum width=2.8cm, align=center},
    level distance=2cm,
    sibling distance=3cm
]

\node (gammoid) at (0,4) [fill=gray!30]{Gammoid};

\node (positroid) at (-4,2) {Positroid};
\node (transversal) at (4,2) [fill=gray!30] {Transversal};

\node (pathcircular) at (4,0) {Path-circular};

\node (laminar) at (-6,-2) [fill=red!40] {Laminar};
\node (colaminar) at (-2,-2) {Colaminar};
\node (multi-path) at (2,-2) {Multi-path};
\node (bicircular) at (6,-2) {Bicircular};

\node (latticepath) at (2,-4) {Lattice path};

\draw[-] (gammoid) -- (positroid);
\draw[-] (gammoid) -- (transversal);

\draw[-] (laminar) -- (positroid);
\draw[-] (colaminar) -- (positroid);
\draw[-] (multi-path) -- (positroid);

\draw[-] (transversal) -- (pathcircular);

\draw[-] (transversal) -- (colaminar);

\draw[dashed] (pathcircular) -- (colaminar);

\draw[-] (bicircular) -- (pathcircular);
\draw[-] (multi-path) -- (pathcircular);

\draw[-] (multi-path) -- (latticepath);

\end{tikzpicture}}
    
    \caption{The classes of matroid considered in this section, ordered by inclusion. 
    The upper bound of \Cref{thm: main result with proof} holds for each class of matroids other than possibly transversal matroids or gammoids. 
    Moreover, they all actually achieve that bound other than the laminar matroids. While we know that not all path-circular matroids are colaminar, we leave the comparability between the two classes open. } 
    \label{fig:poset-of-classes}
\end{figure}

The class of bicircular matroids (with a small technical modification; see \cite{Bicircular-duality}) is closed under minors but not duality, and there are bicircular matroids achieving the equality in \Cref{thm: classes of transversal matroids} \cite[end of Section 6]{Bonin-deMier-LPMs}.

Path-circular matroids are a recently-defined minor-closed class of transversal matroids generalizing both multi-path matroids and bicircular matroids. 
The definition is more involved and can be found in \cite{path-circular}. 
In particular, the class of path-circular matroids contains all multi-path and bicircular matroids. 
Thus, it contains matroids that are not positroids and it also achieves the bound in \Cref{thm: classes of transversal matroids}. We summarize as follows: 

\begin{corollary}
   For all integers $r,\ell\geq 1$, the maximum number of elements of a bicircular matroid or a path-circular matroid with no $U_{2,\ell+2}$-minor is $\ell(r - 1) + 1$. 
\end{corollary}

\section{Future Work} \label{sec: future work}

Theorem \ref{thm: main} motivates several directions for future work.

\subsection{\texorpdfstring{$3$}{3}-connected positroids}
Theorem \ref{thm: main} shows that every simple rank-$r$ positroid with no $U_{2,\ell+2}$-minor and the maximum number of elements is not $3$-connected when $r \ge 3$.
This leads to the following question: what is the maximum number of elements of a simple $3$-connected rank-$r$ positroid with no $U_{2,\ell+2}$-minor?
Brylawski \cite{Brylawski-1971} proved that every $3$-connected matroid with no $M(K_4)$- or $U_{2,4}$-minors has at most $5$ elements, so this question is not interesting when $\ell = 2$.
However, for each $\ell \ge 3$ we will describe a family of $3$-connected $U_{2,\ell+2}$-minor-free positroids that generalize the rank-$r$ whirl, and which we believe achieve this maximum number of elements.

Let $r, \ell$ be integers with $r \ge 2$ and $\ell \ge 3$.
Let $W(r,\ell)$ be the matroid with a basis $B = \{b_1, b_2, \dots, b_r\}$ and $\lfloor \frac{\ell - 1}{2} \rfloor$ elements freely placed (see \cite[page 270]{Oxley-2011}) in the span of $\{b_i, b_{i+1}\}$ for all $i \in [r]$, reading indices modulo $r$.
Let $W(r, \ell)^+$ be obtained from $W(r, \ell)$ by freely placing one point in $\cl(\{b_1, b_2\})$.
Then $W(2, \ell) \cong U_{2,\lfloor \frac{\ell - 1}{2} \rfloor + 2}$, and $W(r, 3)$ and $W(r, 4)$ are both isomorphic to the rank-$r$ whirl $\mathcal W^r$ when $r \ge 3$.
Since the whirl $\mathcal W^r$ is $3$-connected (see \cite[Example 8.4.3]{Oxley-2011}) and $W(r,\ell)$ has $\mathcal W^r$ as a restriction, it follows from \cite[page 295, Exercise 5]{Oxley-2011} that $W(r,\ell)$ and $W(r, \ell)^+$ are $3$-connected.
And it is straightforward to check that the natural cyclic ordering of $E(W(r, \ell))$ (first $b_1$, then all elements in $\cl(\{b_1, b_2\}) - b_2$, then $b_2$, and so on) or $W(r,\ell)^+$ satisfies the condition of Theorem \ref{thm:Bonin's condition}, so $W(r, \ell)$ and $W(r, \ell)^+$ are positroids.
Finally, one can check that if $r \ge 4$ and $e \notin \{b_1, \dots, b_r\}$, then $\si(W(r, \ell)/e) \cong W(r-1, \ell)$, and that $\si(W(r, \ell)/b_i) \cong W(r-1,\ell) \del (\cl(\{b_{i-1},b_i, b_{i+1}) - \{b_{i-1}, b_{i+1}\})$, and from these observations one can prove that $W(r,\ell)$ is $U_{2,\ell+2}$-minor-free, and when $\ell$ is even, $W(r, \ell)^+$ is $U_{2, \ell+2}$-minor-free.
Noting that $|W(r, \ell)| = r + r\lfloor \frac{\ell-1}{2} \rfloor$, we make the following conjecture.

\begin{conjecture} \label{conj: 3-connected positroids}
For integers $r$ and $\ell$ with $r \ge 2$ and $\ell \ge 3$, every $3$-connected rank-$r$ positroid with no $U_{2, \ell+2}$-minor has at most $r\left(\lfloor\frac{\ell}{2}\rfloor+1\right)$ elements. 
\end{conjecture}

As evidence, these bounds hold when $r = 2$, and it is not difficult to prove that they hold when $r = 3$.
By a result of Oxley \cite[Corollary 4.3]{Oxley-1987}, they also hold for $\ell = 3$ in the special case that the positroids are ternary.
Conjecture \ref{conj: 3-connected positroids} may be approachable with techniques from \cite{Oxley-1987} and \cite{Quail-2024}.

\subsection{Other \texorpdfstring{$M(K_4)$}{M(K4)}-minor-free classes}\label{subsec:MK4-free}

We return to transversal matroids and gammoids.
The methods of Section \ref{sec: the main proof} were not sufficient to conclusively resolve the following problem:

\begin{problem} \label{prob: extremal function for gammoids}
Does the bound of Theorem \ref{thm: main} hold for transversal matroids or gammoids?
\end{problem}

Problem \ref{prob: extremal function for gammoids} seems approachable since, as we showed in \Cref{subsec:minor-close-transversal}, all of the matroids of Figures \ref{fig: rank-3 excluded minors} and \ref{fig: rank-4 excluded minor} are not transversal.
The bound of Theorem \ref{thm: main} would be best-possible because one can show via \cite[page 264, Exercise 10]{Oxley-2011} that some (but not all) of the matroids in $\cM_{r,\ell}$ are transversal.
However, the class of transversal matroids is not closed under contraction, and we do make use of contraction several times in the proof of Theorem \ref{thm: main}.
Because gammoids are minors of transversal matroids, it is then natural to consider Problem \ref{prob: extremal function for gammoids} for gammoids.
The difficulty here is that while $M(K_4)$ and $M_6$ are not gammoids \cite[page 644]{Oxley-2011}, one can show via Theorem \ref{thm: Mason result} that all of the other matroids of Figures \ref{fig: rank-3 excluded minors} and \ref{fig: rank-4 excluded minor} are strict gammoids.

\begin{proposition}
Let $M$ be a simple rank-$3$ matroid so that each line of $M$ has at most $3$ points.
Then $M$ is a strict gammoid if and only if for each spanning set $X \subseteq E(M)$ the matroid $M|X$ has at most $|X| - 3$ long lines.
\end{proposition}
\begin{proof}
We recall the parameter $\alpha_M$ from Theorem \ref{thm: Mason result}.
Each rank-$2$ set $X$ satisfies $\alpha_M(X) = 0$ (if $|X| = 2$) or $\alpha_M(X) = 1$ (if $|X| = 3$).
By Theorem \ref{thm: Mason result}, $M$ is a strict gammoid if and only if $\alpha_M(X) \ge 0$ for each spanning set $X \subseteq E$.
Let $X \subseteq E$ be spanning.
Each proper flat $F$ of $M|X$ with at most two points satisfies $\alpha_M(F) = 0$, and each long line $L$ of $M|X$ satisfies $\alpha_M(L) = |L| - 2 = 1$.
Therefore $\sum_{F \in \overline{\cF}(X)}\alpha_M(F)$ is equal to the number of long lines of $M|X$, so $\alpha_M(X)\ge 0$ if and only if $M|X$ has at most $|X| - 3$ long lines.
\end{proof}

So while Problem \ref{prob: extremal function for gammoids} is interesting, it will require new ideas to solve.

We mention that one can extend Problem \ref{prob: extremal function for gammoids} beyond gammoids.
The classes of \emph{strongly base-orderable matroids} \cite[page 435, Exercise 10]{Oxley-2011} and \emph{$k$-base-orderable matroids} \cite{Bonin-Savitsky-2016} are classes that impose strong basis exchange conditions.
Each of these is a minor-closed class that strictly contains the class of gammoids and does not contain $M(K_4)$ \cite[page 435, Exercise 11]{Oxley-2011}, so one approach for Problem \ref{prob: extremal function for gammoids} would be to first determine the extremal bound for the class of strongly base-orderable matroids or $k$-base-orderable matroids.
We leave this an open direction for future work.

\subsection{Oriented matroids}

A rank-$r$ matroid $M$ is \emph{orientable} if there is a function $\chi \colon E(M)^r \to \{-1,0,1\}$ (called a \emph{chirotope}) so that $M$ satisfies a basis exchange property with respect to $\chi$; see \cite{Ziegler-1996} for more details and background.
Since positroids are precisely the positively orientable (meaning that the image of $\chi$ is contained in $\{0,1\}$) matroids  \cite{Ardila-Rincon-Williams-2017}, it is natural to ask the following question.

\begin{problem}
For all positive integers $r$ and $\ell$, what is the maximum number of elements of a simple rank-$r$ orientable matroid with no $U_{2, \ell+2}$-minor?
\end{problem}

Since Dowling geometries over the group of order $2$ are $\bR$-realizable and therefore orientable, there is a lower bound of $2\binom{r}{2} + r$.
And since cyclic Reid geometries are non-orientable (see \cite{Florez-Forge-2007}), it follows from work of Geelen, Nelson, and Walsh \cite{Geelen-Nelson-Walsh-2024} (Theorem 9.4 together with Propositions 10.3 and 10.4) that there is an upper bound
of $\ell\binom{r}{2} + r$ when $r$ is sufficiently large.
To close this gap between the two bounds one must determine the orientability of Dowling geometries over the cyclic group of order $t$ when $t \ge 3$, which may be of independent interest.

We say that a matroid $M$ together with a chirotope $\chi$ is an \textit{oriented matroid} $\mathcal{M}=(M,\chi)$, and that $\mathcal{M}$ is an orientation of $M$.
We say $\mathcal{M}$ is \textit{simple} if $M$ is simple.
One may ask whether every orientation of a sufficiently large simple orientable matroid admits a positively oriented $U_{2,\ell+2}$ minor, which we denote by $U_{2,\ell+2}^+$. 
Oriented matroid restrictions and contractions \cite[Section 3.5]{OrientedMatroidBook} can be described in a straightforward way at the level of the chirotope and, importantly for us, act by (unoriented) restriction and contraction on the underlying matroid. 
That is, if $\mathcal{M}$ is an oriented matroid on ground set $E$ with underlying matroid $M$, then the restriction of $\mathcal{M}$ to $E'\subset E$ has underlying matroid $M|E'$, and similarly for contraction. 
One quickly observes that a negatively oriented matroid (meaning that the image of $\chi$ is contained in $\{-1,0\}$) has negatively oriented minors and hence can never contain a positively oriented $U_{2,\ell+2}$. 
To account for this, we amend our question to determining the minimum size of an oriented matroid which is guaranteed to admit either a positively \textit{or} negatively oriented $U_{2,\ell+2}$-minor (denoted by $U_{2,\ell+2}^-$).

{

\begin{problem} \label{prob: oriented minors}
For all positive integers $r$ and $\ell$, what is the maximum number of elements of a simple rank-$r$ oriented matroid with no $U_{2, \ell+2}^+$- or $U_{2, \ell+2}^-$-minor?
\end{problem}

We can give an upper bound using Ramsey theory.
Let $\cM$ be a simple rank-$r$ oriented matroid with $n$ elements, and let $n_0$ be the Ramsey number $R(\ell + 2, \ell + 2)$.
If $n > (n_0^{r} - 1)/(r - 1)$, then by Theorem \ref{thm: exlude a line} $\cM$ has a minor $\cN$ with underlying matroid $U_{2,n_0}$.
Color an edge of $K_{n_0}$ blue if that pair is positive in $\cN$, and color an edge of $K_{n_0}$ red if that pair is negative in $\cN$.
Since $n_0 = R(\ell + 2, \ell + 2)$, there is some subset of $\ell + 2$ elements of $n_0$ for which either all pairs are positive in $\cN$ (giving a $U_{2, \ell+2}^+$-restriction) or all pairs are negative (giving a $U_{2, \ell+2}^-$-restriction).
Therefore $(n_0^r - 1)/(r - 1)$ is an upper bound for Problem \ref{prob: oriented minors}, but certainly this bound can be improved.
We leave this for future work.

\section*{Acknowledgments}

We thank the anonymous reviewers for their careful reading and insightful comments, which have greatly improved the paper.

\bibliographystyle{abbrv}
\bibliography{References, Matroid_structure_references}

@book {Oxley-2011,
    AUTHOR = {Oxley, James},
     TITLE = {Matroid theory},
    SERIES = {Oxford Graduate Texts in Mathematics},
    VOLUME = {21},
   EDITION = {Second},
 PUBLISHER = {Oxford University Press, Oxford},
      YEAR = {2011},
     PAGES = {xiv+684},
      ISBN = {978-0-19-960339-8},
   MRCLASS = {05-01 (05B35 90C27)},
  MRNUMBER = {2849819},
MRREVIEWER = {Maruti M. Shikare},
       DOI = {10.1093/acprof:oso/9780198566946.001.0001},
       URL = {https://doi.org/10.1093/acprof:oso/9780198566946.001.0001},
}

@article {Bonin-2024,
    AUTHOR = {Bonin, Joseph E.},
     TITLE = {A characterization of positroids, with applications to
              amalgams and excluded minors},
   JOURNAL = {European J. Combin.},
  FJOURNAL = {European Journal of Combinatorics},
    VOLUME = {122},
      YEAR = {2024},
     PAGES = {Paper No. 104040, 33},
      ISSN = {0195-6698},
   MRCLASS = {05B35},
  MRNUMBER = {4782415},
MRREVIEWER = {Martin Kochol},
       DOI = {10.1016/j.ejc.2024.104040},
       URL = {https://doi.org/10.1016/j.ejc.2024.104040},
}

@article {Bonin-Savitsky-2016,
    AUTHOR = {Bonin, Joseph E. and Savitsky, Thomas J.},
     TITLE = {An infinite family of excluded minors for strong
              base-orderability},
   JOURNAL = {Linear Algebra Appl.},
  FJOURNAL = {Linear Algebra and its Applications},
    VOLUME = {488},
      YEAR = {2016},
     PAGES = {396--429},
      ISSN = {0024-3795},
   MRCLASS = {05B35},
  MRNUMBER = {3419793},
       DOI = {10.1016/j.laa.2015.09.055},
       URL = {https://doi.org/10.1016/j.laa.2015.09.055},
}

@article {Chidiac-Hochstattler-2024,
    AUTHOR = {Chidiac, Lamar and Hochst\"{a}ttler, Winfried},
     TITLE = {Positroids are 3-colorable},
   JOURNAL = {Studia Sci. Math. Hungar.},
  FJOURNAL = {Studia Scientiarum Mathematicarum Hungarica. Combinatorics,
              Geometry and Topology (CoGeTo)},
    VOLUME = {61},
      YEAR = {2024},
    NUMBER = {2},
     PAGES = {147--160},
      ISSN = {0081-6906},
   MRCLASS = {05C15},
  MRNUMBER = {4815460},
MRREVIEWER = {Shaopu Zhang},
}

@article {Geelen-Nelson-2010,
    AUTHOR = {Geelen, Jim and Nelson, Peter},
     TITLE = {The number of points in a matroid with no {$n$}-point line as
              a minor},
   JOURNAL = {J. Combin. Theory Ser. B},
  FJOURNAL = {Journal of Combinatorial Theory. Series B},
    VOLUME = {100},
      YEAR = {2010},
    NUMBER = {6},
     PAGES = {625--630},
      ISSN = {0095-8956},
   MRCLASS = {05B35 (05C83 51A45)},
  MRNUMBER = {2718682},
MRREVIEWER = {Stefan H. M. van Zwam},
       DOI = {10.1016/j.jctb.2010.06.001},
       URL = {https://doi.org/10.1016/j.jctb.2010.06.001},
}

@article {Geelen-Nelson-Walsh-2024,
    AUTHOR = {Geelen, Jim and Nelson, Peter and Walsh, Zach},
     TITLE = {Excluding a line from complex-representable matroids},
   JOURNAL = {Mem. Amer. Math. Soc.},
  FJOURNAL = {Memoirs of the American Mathematical Society},
    VOLUME = {303},
      YEAR = {2024},
    NUMBER = {1523},
     PAGES = {v+91},
      ISSN = {0065-9266},
      ISBN = {978-1-4704-7187-3; 978-1-4704-7992-3},
   MRCLASS = {05B35},
  MRNUMBER = {4837615},
MRREVIEWER = {Joseph E. Bonin},
       DOI = {10.1090/memo/1523},
       URL = {https://doi.org/10.1090/memo/1523},
}

@article {Ingleton-Piff-1973,
    AUTHOR = {Ingleton, A. W. and Piff, M. J.},
     TITLE = {Gammoids and transversal matroids},
   JOURNAL = {J. Combinatorial Theory Ser. B},
  FJOURNAL = {Journal of Combinatorial Theory. Series B},
    VOLUME = {15},
      YEAR = {1973},
     PAGES = {51--68},
      ISSN = {0095-8956},
   MRCLASS = {05B35},
  MRNUMBER = {329936},
MRREVIEWER = {D. J. A. Welsh},
       DOI = {10.1016/0095-8956(73)90031-2},
       URL = {https://doi.org/10.1016/0095-8956(73)90031-2},
}

@article {Kung-1988,
    AUTHOR = {Kung, Joseph P. S.},
     TITLE = {The long-line graph of a combinatorial geometry. {I}.
              {E}xcluding {$M(K_4)$} and the {$(q+2)$}-point line as minors},
   JOURNAL = {Quart. J. Math. Oxford Ser. (2)},
  FJOURNAL = {The Quarterly Journal of Mathematics. Oxford. Second Series},
    VOLUME = {39},
      YEAR = {1988},
    NUMBER = {154},
     PAGES = {223--234},
      ISSN = {0033-5606},
   MRCLASS = {51D20 (05C35)},
  MRNUMBER = {947503},
MRREVIEWER = {J. Andr\'{e}},
       DOI = {10.1093/qmath/39.2.223},
       URL = {https://doi.org/10.1093/qmath/39.2.223},
}

@incollection {Kung-1993,
    AUTHOR = {Kung, Joseph P. S.},
     TITLE = {Extremal matroid theory},
 BOOKTITLE = {Graph structure theory ({S}eattle, {WA}, 1991)},
    SERIES = {Contemp. Math.},
    VOLUME = {147},
     PAGES = {21--61},
 PUBLISHER = {Amer. Math. Soc., Providence, RI},
      YEAR = {1993},
   MRCLASS = {05B35 (06C10)},
  MRNUMBER = {1224696},
MRREVIEWER = {James G. Oxley},
       DOI = {10.1090/conm/147/01164},
       URL = {https://doi.org/10.1090/conm/147/01164},
}

@article {Mason-1972,
    AUTHOR = {Mason, J. H.},
     TITLE = {On a class of matroids arising from paths in graphs},
   JOURNAL = {Proc. London Math. Soc. (3)},
  FJOURNAL = {Proceedings of the London Mathematical Society. Third Series},
    VOLUME = {25},
      YEAR = {1972},
     PAGES = {55--74},
      ISSN = {0024-6115},
   MRCLASS = {05B35},
  MRNUMBER = {311496},
MRREVIEWER = {George W. Dinolt},
       DOI = {10.1112/plms/s3-25.1.55},
       URL = {https://doi.org/10.1112/plms/s3-25.1.55},
}

@article {Oxley-1987,
    AUTHOR = {Oxley, James G.},
     TITLE = {A characterization of the ternary matroids with no
              {$M(K_4)$}-minor},
   JOURNAL = {J. Combin. Theory Ser. B},
  FJOURNAL = {Journal of Combinatorial Theory. Series B},
    VOLUME = {42},
      YEAR = {1987},
    NUMBER = {2},
     PAGES = {212--249},
      ISSN = {0095-8956},
   MRCLASS = {05B35},
  MRNUMBER = {884255},
MRREVIEWER = {Ulrich Faigle},
       DOI = {10.1016/0095-8956(87)90041-4},
       URL = {https://doi.org/10.1016/0095-8956(87)90041-4},
}

@article {Pendavingh-vanderPol-2015,
    AUTHOR = {Pendavingh, R. A. and van der Pol, J. G.},
     TITLE = {Counting matroids in minor-closed classes},
   JOURNAL = {J. Combin. Theory Ser. B},
  FJOURNAL = {Journal of Combinatorial Theory. Series B},
    VOLUME = {111},
      YEAR = {2015},
     PAGES = {126--147},
      ISSN = {0095-8956},
   MRCLASS = {05B35},
  MRNUMBER = {3315602},
MRREVIEWER = {Eva Ferrara Dentice},
       DOI = {10.1016/j.jctb.2014.10.001},
       URL = {https://doi.org/10.1016/j.jctb.2014.10.001},
}

@article {Quail-Rombach-2025,
    AUTHOR = {Quail, Jeremy and Rombach, Puck},
     TITLE = {Positroid envelopes and graphic positroids},
   JOURNAL = {Comb. Theory},
  FJOURNAL = {Combinatorial Theory},
    VOLUME = {5},
      YEAR = {2025},
    NUMBER = {3},
     PAGES = {Paper No. 1, 31},
      ISBN = {},
   MRCLASS = {05B35},
  MRNUMBER = {4962115},
}

@article {Sims-1977,
    AUTHOR = {Sims, Julie A.},
     TITLE = {A complete class of matroids},
   JOURNAL = {Quart. J. Math. Oxford Ser. (2)},
  FJOURNAL = {The Quarterly Journal of Mathematics. Oxford. Second Series},
    VOLUME = {28},
      YEAR = {1977},
    NUMBER = {112},
     PAGES = {449--451},
      ISSN = {0033-5606},
   MRCLASS = {05B35},
  MRNUMBER = {485464},
MRREVIEWER = {Thomas Brylawski},
       DOI = {10.1093/qmath/28.4.449},
       URL = {https://doi.org/10.1093/qmath/28.4.449},
}

@article {vanderPol-2023,
    AUTHOR = {van der Pol, Jorn},
     TITLE = {Almost every matroid has an {$M(K_4)$}- or a
              {$\mathcal{W}^3$}-minor},
   JOURNAL = {Electron. J. Combin.},
  FJOURNAL = {Electronic Journal of Combinatorics},
    VOLUME = {30},
      YEAR = {2023},
    NUMBER = {4},
     PAGES = {Paper No. 4.23, 10},
   MRCLASS = {05B35 (05A16)},
  MRNUMBER = {4668198},
MRREVIEWER = {Andr\'{a}s Recski},
       DOI = {10.37236/11946},
       URL = {https://doi.org/10.37236/11946},
}

@article {Brylawski-1971,
    AUTHOR = {Brylawski, Thomas H.},
     TITLE = {A combinatorial model for series-parallel networks},
   JOURNAL = {Trans. Amer. Math. Soc.},
  FJOURNAL = {Transactions of the American Mathematical Society},
    VOLUME = {154},
      YEAR = {1971},
     PAGES = {1--22},
      ISSN = {0002-9947},
   MRCLASS = {05.35 (50.00)},
  MRNUMBER = {288039},
MRREVIEWER = {F. Harary},
       DOI = {10.2307/1995423},
       URL = {https://doi.org/10.2307/1995423},
}

@unpublished{Quail-2024,
  title =	 {Ternary and quaternary positroids},
  author =	 {J. Quail},
  note = {ar{X}iv:2403.06956},
  year =	 2024
}

@unpublished{Park-2023,
  title =	 {EXCLUDED MINORS OF INTERVAL POSITROIDS THAT
ARE PAVING MATROIDS},
  author =	 {H. Park},
  note = {ar{X}iv:2312.03525},
  year =	 2023
}

@article {Blum-2001,
    AUTHOR = {Blum, Stefan},
     TITLE = {Base-sortable matroids and {K}oszulness of semigroup rings},
   JOURNAL = {European J. Combin.},
  FJOURNAL = {European Journal of Combinatorics},
    VOLUME = {22},
      YEAR = {2001},
    NUMBER = {7},
     PAGES = {937--951},
      ISSN = {0195-6698},
   MRCLASS = {05B35 (16S38)},
  MRNUMBER = {1857256},
MRREVIEWER = {M. Stef\u{a}nescu},
       DOI = {10.1006/eujc.2001.0516},
       URL = {https://doi.org/10.1006/eujc.2001.0516},
}

@article {Ardila-Rincon-Williams-2017,
    AUTHOR = {Ardila, Federico and Rinc\'{o}n, Felipe and Williams, Lauren},
     TITLE = {Positively oriented matroids are realizable},
   JOURNAL = {J. Eur. Math. Soc. (JEMS)},
  FJOURNAL = {Journal of the European Mathematical Society (JEMS)},
    VOLUME = {19},
      YEAR = {2017},
    NUMBER = {3},
     PAGES = {815--833},
      ISSN = {1435-9855},
   MRCLASS = {52C40 (05B35 14M15)},
  MRNUMBER = {3612868},
MRREVIEWER = {Joseph E. Bonin},
       DOI = {10.4171/JEMS/680},
       URL = {https://doi.org/10.4171/JEMS/680},
}

@article {Florez-Forge-2007,
    AUTHOR = {Fl\'{o}rez, Rigoberto and Forge, David},
     TITLE = {Minimal non-orientable matroids in a projective plane},
   JOURNAL = {J. Combin. Theory Ser. A},
  FJOURNAL = {Journal of Combinatorial Theory. Series A},
    VOLUME = {114},
      YEAR = {2007},
    NUMBER = {1},
     PAGES = {175--183},
      ISSN = {0097-3165},
   MRCLASS = {05B35 (05B25 51D20 51E15 52C35)},
  MRNUMBER = {2276967},
MRREVIEWER = {Joseph Kung},
       DOI = {10.1016/j.jcta.2006.03.002},
       URL = {https://doi.org/10.1016/j.jcta.2006.03.002},
}

@article {Ziegler-1996,
    AUTHOR = {Ziegler, G\"{u}nter M.},
     TITLE = {Oriented matroids today},
   JOURNAL = {Electron. J. Combin.},
  FJOURNAL = {Electronic Journal of Combinatorics},
    VOLUME = {3},
      YEAR = {1996},
    NUMBER = {1},
     PAGES = {Dynamic Survey 4, 39},
   MRCLASS = {52B05 (52B40 52B55 52C35 53B35)},
  MRNUMBER = {1668047},
MRREVIEWER = {Ilda P. F. da Silva},
       URL = {http://www.combinatorics.org/Surveys/index.html}
}

@unpublished{postnikov,
      title={Total positivity, Grassmannians, and networks}, 
      author={Alexander Postnikov},

  note = {ar{X}iv:math/0609764},
  year =	 2006
}

@article{AHT,
  author       = {Nima Arkani-Hamed and Jaroslav Trnka},
  title        = {The Amplituhedron},
  journal      = {J. High Energy Phys.},
  year         = {2014},
 volume         = {2014},
  number       = {10},
  pages        = {030},
  doi          = {10.1007/JHEP10(2014)030},
  url          = {https://doi.org/10.1007/JHEP10(2014)030}
}

@article {Williams05,
    AUTHOR = {Williams, Lauren K.},
     TITLE = {Enumeration of totally positive {G}rassmann cells},
   JOURNAL = {Adv. Math.},
  FJOURNAL = {Advances in Mathematics},
    VOLUME = {190},
      YEAR = {2005},
    NUMBER = {2},
     PAGES = {319--342},
      ISSN = {0001-8708,1090-2082},
   MRCLASS = {05E15 (05A15 14M15)},
  MRNUMBER = {2102660},
MRREVIEWER = {Gregory\ S.\ Warrington},
       DOI = {10.1016/j.aim.2004.01.003},
       URL = {https://doi.org/10.1016/j.aim.2004.01.003},
}

@article {ARW16,
    AUTHOR = {Ardila, Federico and Rinc\'on, Felipe and Williams, Lauren},
     TITLE = {Positroids and non-crossing partitions},
   JOURNAL = {Trans. Amer. Math. Soc.},
  FJOURNAL = {Transactions of the American Mathematical Society},
    VOLUME = {368},
      YEAR = {2016},
    NUMBER = {1},
     PAGES = {337--363},
      ISSN = {0002-9947,1088-6850},
   MRCLASS = {05B35 (05A15 05A19 46L53)},
  MRNUMBER = {3413866},
MRREVIEWER = {Anna\ de Mier},
       DOI = {10.1090/tran/6331},
       URL = {https://doi.org/10.1090/tran/6331},
}

@book{OrientedMatroidBook,
  author    = {Anders Bj{\"o}rner and Michel Las Vergnas and Bernd Sturmfels
               and Neil White and G{\"u}nter M. Ziegler},
  title     = {Oriented Matroids},
  edition   = {2},
  series    = {Encyclopedia of Mathematics and its Applications},
  volume    = {46},
  publisher = {Cambridge University Press},
  address   = {Cambridge},
  year      = {1999},
  isbn      = {9780521777506},
  doi       = {10.1017/CBO9780511586507}
}

@article {positroid-is-base-orderable,
    AUTHOR = {Lam, Thomas and Postnikov, Alexander},
     TITLE = {Polypositroids},
   JOURNAL = {Forum Math. Sigma},
  FJOURNAL = {Forum of Mathematics. Sigma},
    VOLUME = {12},
      YEAR = {2024},
     PAGES = {Paper No. e42, 67},
      ISSN = {2050-5094},
   MRCLASS = {52B40 (05B35 20F55 52B12)},
  MRNUMBER = {4718184},
MRREVIEWER = {Winfried\ Hochst\"attler},
       DOI = {10.1017/fms.2024.11},
       URL = {https://doi.org/10.1017/fms.2024.11},
}

@article {bicircular,
    AUTHOR = {Chun, Deborah and Moss, Tyler and Slilaty, Daniel and Zhou,
              Xiangqian},
     TITLE = {Bicircular matroids representable over {$GF(4)$} or {$GF(5)$}},
   JOURNAL = {Discrete Math.},
  FJOURNAL = {Discrete Mathematics},
    VOLUME = {339},
      YEAR = {2016},
    NUMBER = {9},
     PAGES = {2239--2248},
      ISSN = {0012-365X,1872-681X},
   MRCLASS = {05B35},
  MRNUMBER = {3512338},
MRREVIEWER = {Joseph\ E.\ Bonin},
       DOI = {10.1016/j.disc.2016.03.017},
       URL = {https://doi.org/10.1016/j.disc.2016.03.017},
}

@article {laminar-matroids,
    AUTHOR = {Fife, Tara and Oxley, James},
     TITLE = {Laminar matroids},
   JOURNAL = {European J. Combin.},
  FJOURNAL = {European Journal of Combinatorics},
    VOLUME = {62},
      YEAR = {2017},
     PAGES = {206--216},
      ISSN = {0195-6698,1095-9971},
   MRCLASS = {05B35},
  MRNUMBER = {3621735},
MRREVIEWER = {Edward\ B.\ Swartz},
       DOI = {10.1016/j.ejc.2017.01.002},
       URL = {https://doi.org/10.1016/j.ejc.2017.01.002},
}

@article{Bonin-deMier-LPMs,
    AUTHOR = {Bonin, Joseph E. and de Mier, Anna},
     TITLE = {Lattice path matroids: structural properties},
   JOURNAL = {European J. Combin.},
  FJOURNAL = {European Journal of Combinatorics},
    VOLUME = {27},
      YEAR = {2006},
    NUMBER = {5},
     PAGES = {701--738},
      ISSN = {0195-6698,1095-9971},
   MRCLASS = {05B35},
  MRNUMBER = {2215428},
MRREVIEWER = {Neil\ L.\ White},
       DOI = {10.1016/j.ejc.2005.01.008},
       URL = {https://doi.org/10.1016/j.ejc.2005.01.008},
}

@unpublished{path-circular,
      title={A new minor-closed class of transversal matroids}, 
      author={Gerry Toft},

  note = {ar{X}iv:2511.13089},
  year =	 2025
}

@article {BdMN,
    AUTHOR = {Bonin, Joseph and de Mier, Anna and Noy, Marc},
     TITLE = {Lattice path matroids: enumerative aspects and {T}utte
              polynomials},
   JOURNAL = {J. Combin. Theory Ser. A},
  FJOURNAL = {Journal of Combinatorial Theory. Series A},
    VOLUME = {104},
      YEAR = {2003},
    NUMBER = {1},
     PAGES = {63--94},
      ISSN = {0097-3165,1096-0899},
   MRCLASS = {05B35 (05A10 05A15)},
  MRNUMBER = {2018421},
MRREVIEWER = {Joseph\ Kung},
       DOI = {10.1016/S0097-3165(03)00122-5},
       URL = {https://doi.org/10.1016/S0097-3165(03)00122-5},
}

@article {Oh,
    AUTHOR = {Oh, Suho},
     TITLE = {Positroids and {S}chubert matroids},
   JOURNAL = {J. Combin. Theory Ser. A},
  FJOURNAL = {Journal of Combinatorial Theory. Series A},
    VOLUME = {118},
      YEAR = {2011},
    NUMBER = {8},
     PAGES = {2426--2435},
      ISSN = {0097-3165,1096-0899},
   MRCLASS = {05B35 (52B40)},
  MRNUMBER = {2834184},
MRREVIEWER = {Anna\ de Mier},
       DOI = {10.1016/j.jcta.2011.06.006},
       URL = {https://doi.org/10.1016/j.jcta.2011.06.006},
}

@article {Matthews-Bicircular,
    AUTHOR = {Matthews, Laurence R.},
     TITLE = {Bicircular matroids},
   JOURNAL = {Quart. J. Math. Oxford Ser. (2)},
  FJOURNAL = {The Quarterly Journal of Mathematics. Oxford. Second Series},
    VOLUME = {28},
      YEAR = {1977},
    NUMBER = {110},
     PAGES = {213--227},
      ISSN = {0033-5606,1464-3847},
   MRCLASS = {05B35},
  MRNUMBER = {505702},
MRREVIEWER = {J.\ Sheehan},
       DOI = {10.1093/qmath/28.2.213},
       URL = {https://doi.org/10.1093/qmath/28.2.213},
}

@unpublished{Positroids-and-Transversals,
      title={Transversal and Paving Positroids}, 
      author={John Machacek and George D. Nasr},
  note ={ar{X}iv:2401.02053},
      year={2024}
}

@article{LPM-quotient,
  title={Lattice path matroids and quotients},
  author={Benedetti-Vel{\'a}squez, Carolina and Knauer, Kolja},
  journal={Combinatorica},
  volume={44},
  number={3},
  pages={621--650},
  year={2024},
  publisher={Springer}
}

@inproceedings {Secretary-laminar-2011,
    AUTHOR = {Im, Sungjin and Wang, Yajun},
     TITLE = {Secretary problems: laminar matroid and interval scheduling},
 BOOKTITLE = {Proceedings of the {T}wenty-{S}econd {A}nnual {ACM}-{SIAM}
              {S}ymposium on {D}iscrete {A}lgorithms},
     PAGES = {1265--1274},
 PUBLISHER = {SIAM, Philadelphia, PA},
      YEAR = {2011},
   MRCLASS = {91A60 (05B35)},
  MRNUMBER = {2858398},
}

@incollection {Secretary-laminar-2013,
    AUTHOR = {Jaillet, Patrick and Soto, Jos\'{e} A. and Zenklusen, Rico},
     TITLE = {Advances on matroid secretary problems: free order model and
              laminar case},
 BOOKTITLE = {Integer programming and combinatorial optimization},
    SERIES = {Lecture Notes in Comput. Sci.},
    VOLUME = {7801},
     PAGES = {254--265},
 PUBLISHER = {Springer, Heidelberg},
      YEAR = {2013},
      ISBN = {978-3-642-36694-9; 978-3-642-36693-2},
   MRCLASS = {91A60 (05B35 90C27)},
  MRNUMBER = {3085551},
       DOI = {10.1007/978-3-642-36694-9\{_}22}

@article {Generalized-laminar,
    AUTHOR = {Fife, Tara and Oxley, James},
     TITLE = {Generalized laminar matroids},
   JOURNAL = {European J. Combin.},
  FJOURNAL = {European Journal of Combinatorics},
    VOLUME = {79},
      YEAR = {2019},
     PAGES = {111--122},
      ISSN = {0195-6698,1095-9971},
   MRCLASS = {05B35},
  MRNUMBER = {3899087},
MRREVIEWER = {Daryl\ Funk},
       DOI = {10.1016/j.ejc.2018.12.005},
       URL = {https://doi.org/10.1016/j.ejc.2018.12.005},
}

@unpublished {Finkelstein_thesis,
    AUTHOR = {Finkelstein, L.},
     TITLE = {Two algorithms for the matroid secretary problem},
      NOTE = {Master's thesis, Technion, Israel Institute of Technology},
      YEAR = {2011},
}

@unpublished{LPM-Ehrhart,
  title={Skew shapes, Ehrhart positivity and beyond},
  author={Ferroni, Luis and Morales, Alejandro H and Panova, Greta},
  note ={ar{X}iv:2503.16403},
  year={2025}
}

@article {Multipath,
    AUTHOR = {Bonin, Joseph E. and Gim\'{e}nez, Omer},
     TITLE = {Multi-path matroids},
   JOURNAL = {Combin. Probab. Comput.},
  FJOURNAL = {Combinatorics, Probability and Computing},
    VOLUME = {16},
      YEAR = {2007},
    NUMBER = {2},
     PAGES = {193--217},
      ISSN = {0963-5483,1469-2163},
   MRCLASS = {05B35 (05C83)},
  MRNUMBER = {2298809},
MRREVIEWER = {Martin\ Kochol},
       DOI = {10.1017/S0963548306007942},
       URL = {https://doi.org/10.1017/S0963548306007942},
}

@article {Bicircular-1972,
    AUTHOR = {Simoes-Pereira, J. M. S.},
     TITLE = {On subgraphs as matroid cells},
   JOURNAL = {Math. Z.},
  FJOURNAL = {Mathematische Zeitschrift},
    VOLUME = {127},
      YEAR = {1972},
     PAGES = {315--322},
      ISSN = {0025-5874,1432-1823},
   MRCLASS = {05B35 (05C99)},
  MRNUMBER = {317973},
MRREVIEWER = {A.\ J.\ Schwenk},
       DOI = {10.1007/BF01111390},
       URL = {https://doi.org/10.1007/BF01111390},
}

@article {Exclude-uniform-matroid,
    AUTHOR = {Geelen, Jim},
     TITLE = {Some open problems on excluding a uniform matroid},
   JOURNAL = {Adv. in Appl. Math.},
  FJOURNAL = {Advances in Applied Mathematics},
    VOLUME = {41},
      YEAR = {2008},
    NUMBER = {4},
     PAGES = {628--637},
      ISSN = {0196-8858,1090-2074},
   MRCLASS = {05B35 (05C83 15A33)},
  MRNUMBER = {2459453},
MRREVIEWER = {Haidong\ Wu},
       DOI = {10.1016/j.aam.2008.05.002},
       URL = {https://doi.org/10.1016/j.aam.2008.05.002},
}

@article {Structure-theory,
    AUTHOR = {Geelen, Jim and Gerards, Bert and Whittle, Geoff},
     TITLE = {The highly connected matroids in minor-closed classes},
   JOURNAL = {Ann. Comb.},
  FJOURNAL = {Annals of Combinatorics},
    VOLUME = {19},
      YEAR = {2015},
    NUMBER = {1},
     PAGES = {107--123},
      ISSN = {0218-0006,0219-3094},
   MRCLASS = {05B35},
  MRNUMBER = {3319863},
MRREVIEWER = {Joseph\ E.\ Bonin},
       DOI = {10.1007/s00026-015-0251-3},
       URL = {https://doi.org/10.1007/s00026-015-0251-3},
}

@article {Zaslavsky1991,
    AUTHOR = {Zaslavsky, Thomas},
     TITLE = {Biased graphs. {II}. {T}he three matroids},
   JOURNAL = {J. Combin. Theory Ser. B},
  FJOURNAL = {Journal of Combinatorial Theory. Series B},
    VOLUME = {51},
      YEAR = {1991},
    NUMBER = {1},
     PAGES = {46--72},
      ISSN = {0095-8956},
   MRCLASS = {05B35 (05C99)},
  MRNUMBER = {1088626},
MRREVIEWER = {J. M. S. Sim\~{o}es-Pereira},
       DOI = {10.1016/0095-8956(91)90005-5},
       URL = {https://doi.org/10.1016/0095-8956(91)90005-5},
}

@article {Bicircular-duality,
    AUTHOR = {Sivaraman, Vaidy and Slilaty, Daniel},
     TITLE = {The family of bicircular matroids closed under duality},
   JOURNAL = {Graphs Combin.},
  FJOURNAL = {Graphs and Combinatorics},
    VOLUME = {38},
      YEAR = {2022},
    NUMBER = {1},
     PAGES = {Paper No. 24, 20},
      ISSN = {0911-0119,1435-5914},
   MRCLASS = {05B35},
  MRNUMBER = {4356265},
MRREVIEWER = {Laura\ Bertani},
       DOI = {10.1007/s00373-021-02413-7},
       URL = {https://doi.org/10.1007/s00373-021-02413-7},
}

@article {Lattice-path-minors2010,
    AUTHOR = {Bonin, Joseph E.},
     TITLE = {Lattice path matroids: the excluded minors},
   JOURNAL = {J. Combin. Theory Ser. B},
  FJOURNAL = {Journal of Combinatorial Theory. Series B},
    VOLUME = {100},
      YEAR = {2010},
    NUMBER = {6},
     PAGES = {585--599},
      ISSN = {0095-8956,1096-0902},
   MRCLASS = {05B35 (05C83)},
  MRNUMBER = {2718679},
MRREVIEWER = {Anna\ de Mier},
       DOI = {10.1016/j.jctb.2010.05.001},
       URL = {https://doi.org/10.1016/j.jctb.2010.05.001},
}

\end{document}